\documentclass[ 12pt]{elsarticle}

\usepackage{graphicx, subfigure}
\usepackage{amsmath,amscd,amssymb,amsfonts, amsthm, epstopdf}
\usepackage[arrow, matrix, curve]{xy}
 \usepackage[hang,small,bf]{caption}
 \usepackage{hyperref}

\numberwithin{equation}{section}

\DeclareMathOperator{\im}{im}
\DeclareMathOperator{\sgn}{sgn}

\DeclareMathOperator{\e}{e}

\DeclareMathOperator{\coker}{coker}

\begin{document}
\theoremstyle{plain}
\newtheorem{thm}{Theorem}[section]
\newtheorem{lem}[thm]{Lemma}
\newtheorem{prop}[thm]{Proposition}
\newtheorem{coll}[thm]{Corollary}
\newtheorem{conj}{Conjecture}[section]
\theoremstyle{definition}
\newtheorem{defn}{Definition}[section]
\newtheorem{exmp}{Example}[section]

\theoremstyle{remark}
\newtheorem{rem}{Remark}[section]
\newtheorem*{note}{Note}

\newcommand{\rb}{\mathbb{R}}
\newcommand{\lup}{\Delta_{i}^{up}}
\newcommand{\ldown}{\Delta_{i}^{down}}
\newcommand{\s}{\mathbf{s}}
\newcommand{\dote}{\stackrel{\circ}{=}}
\newcommand{\dotu}{\stackrel{\circ}{\cup}}
\begin{frontmatter}

\title{Interlacing inequalities for eigenvalues of discrete Laplace operators}
\author{Danijela Horak}
\address{Max Planck Institute for Mathematics in the Sciences,
  Inselstrasse 22, D-04103 Leipzig, Germany}
\author{J\"urgen Jost}
\address{Max Planck Institute for Mathematics in the Sciences,
  Inselstrasse 22, D-04103 Leipzig, Germany\\
Department of Mathematics and Computer Science, Leipzig
  University, D-04091 Leipzig, Germany\\
Santa Fe Institute for the Sciences of Complexity, Santa Fe, NM 87501,
USA}
\date{}
\begin{abstract}
The term {\it interlacing} refers to  systematic inequalities
between the sequences of eigenvalues of two operators defined on
objects related by a specific operation. In particular, knowledge of
the spectrum of one of the objects then implies eigenvalue bounds for
the other one.\\
In this paper, we therefore develop topological arguments in order to
derive such analytical inequalities. 
We investigate, in a general and systematic manner, interlacing of
spectra  for weighted simplicial complexes with arbitrary
weights. This enables us to control the spectral effects of operations
like deletion of a subcomplex, collapsing and contraction of  a
simplex,  coverings and simplicial maps, for  absolute and relative
Laplacians. It turns out that many  well-known results from graph
theory become special cases of our general results and consequently
admit improvements and generalizations. In particular, we
 derive a number of effective eigenvalue bounds.
 \end{abstract}
 \end{frontmatter}
\section{Introduction}
\label{section 1}
Spectra of Laplace operators typically encode important geometric
information about the space on which they are defined. This raises the
question of how to control the effect on the spectrum when we perform
some standard operation on that underlying space, like cutting out, contracting or
adding a part or taking a local covering. A good quantitative answer
to such a question will be very useful when we can start from some
space whose spectrum is explicitly known or at least tightly
controlled, and then pass by such operations to other spaces whose
spectrum we would like to know. In fact, as we shall explore in this
paper, there exist some general such relations between the spectra of
spaces related by specific operations. More precisely, the eigenvalues
of such spaces control each other, with inequalities like 
$$ \lambda_{k-\kappa_1}\le \theta_k \le \lambda_{k+\kappa_2}$$ for the
corresponding eigenvalues $\lambda_k$ and $\theta_k$, resp., with
integers $\kappa_1, \kappa_2$ that only depend on certain 
topological characteristics of the spaces and operations involved, but
are independent of the index $k$. See Theorem \ref{main theorem}. The main point of this paper then is
a systematic scheme how to derive such 
{\it analytical} inequalities from {\it topological}
considerations. In particular, this provides a unifying perspective on
various special results scattered throughout the literature.\\\\
We shall work here within the framework of generalized graphs, i.e.,  simplicial
complexes; let us start with a brief overview of necessary definitions and theorems.
An \emph{abstract simplicial complex} $K$ on a finite set $V$ is a collection of subsets of $V$ which is closed under inclusion. 
An $n$-face or an $n$-simplex of $K$ is an element of cardinality $n+1$, and the set of all $n$-faces of complex $K$ is denoted by 
$S_{n}(K)$. Two $n+1$-simplices sharing an $n$-face are called \emph{$n$-down neighbours}, and  two $n$-simplices belonging to a boundary of the same $(n+1)$-face are called  \emph{$(n+1)$-up neighbours}.
A simplicial complex $K$ is \emph{$(n+1)$-path connected}, if 
for every  pair of $(n+1)$-simplices $\bar{F}$, $\bar{F}'$ there exists a sequence of $(n+1)$-simplices $\bar{F}=\bar{F}_{1},\bar{F}_{2}\ldots,\bar{F}_{k}=\bar{F}'$, such that 
any two neighbouring ones are $n$-down neighbours. \\
\emph{An oriented simplex} $[F]$ is a simplex $F$ together with an
ordering of its vertices. Two orderings of the vertices are said to
determine the same/opposite orientation when related  by an even/odd permutation.\\
The $n$-th chain group with coefficients in $\rb$, denoted by $C_{n}(K,\rb)$, is   a vector  space  over  $\rb$ generated by oriented 
$n$-faces  of  $K$ modulo the relation $[F_{1}]+[F_{2}]=0$, when
$[F_{1}]$ and $[F_{2}]$ are two different orientations of
the same $n$-simplex.
Elements of $C_{n}(K,\rb)$  are formal $\rb$-linear  sums  of oriented $n$-dimensional faces $[F]$ of $K$.
Let $C^{n}(K,\rb)$ denote $n$-th cochain group, i.e., the dual of the vector space $C_{n}(K,\rb)$.
A basis  of $C^{n}(K,\rb)$ is  given by the set of  \emph{elementary cochains} $\{e_{[F]}\mid F \in S_{n}(K,\rb))\}$ with
$$e_{[F]}([F'])=\left\{  \begin{array}{ll}
1 & \textrm{ if } [F']=[F], \\
0 & \textrm{ otherwise. }
\end{array}
\right. $$
The coboundary operator  $\delta_{n}:C^{n}(K,\rb)\rightarrow C^{n+1}(K,\rb)$ is the linear map 
$$
(\delta_{n}f)([v_{0},\ldots,v_{n+1}]) =\sum_{i=0}^{n+1}(-1)^{i}f([v_{0},\ldots,\hat{v}_{i}\ldots v_{n+1}]),
$$
where  $\hat{v}_{i}$ denotes that the vertex $v_{i}$ has been omitted.

After introducing scalar products on  cochain vector spaces we define  adjoints $\delta_{n}^{*}:C^{n+1}(K,\rb) \rightarrow C^{n}(K,\rb)$ of the coboundary maps $\delta_{n}$
 and the combinatorial Laplace operator
$$
\mathcal{L}_{n}(K)=\delta_{n}^{*}\delta_{n}+ \delta_{n-1}\delta_{n-1}^{*}.
$$
The operators $\delta_{n}^{*}\delta_{n}$ and $ \delta_{n-1}\delta_{n-1}^{*}$  are called the \emph{$n$-up}  and \emph{$n$-down Laplace operators}  and are denoted by $\mathcal{L}_{n}^{up}(K)$ and $\mathcal{L}_{n}^{down}(K)$, respectively.

Scalar products are commonly chosen such  that  elementary cochains form an orthonormal basis.
In other words,
\begin{equation}
\label{natural scalar product}
(f,g)_{C^{n}}=\sum_{F\in S_{n}(K)}f([F])g([F]),
\end{equation}
where $(\: ,\: )_{C^{n}}$ denotes a scalar product defined on the vector space $C^{n}(K,\rb)$.
However, it is possible and, as we shall see, advantageous, to consider inner products on cochain spaces,  different from (\ref{natural scalar product}).
In particular, we will consider a scalar product that turns elementary
cochains into an orthogonal, but not necessarily orthonormal, basis.
In fact, a positive real valued weight function  $w:\bigcup_{i}S_{i}(K)\rightarrow \rb^{+}$  uniquely determines   scalar products on cochain groups, i.e.,
$$
(f,g)_{C^{n}}=\sum_{F\in S_{n}(K)} w(F)f([F])g([F]).
$$
A simplicial complex $K$ together with a weight function $w$ forms an ordered pair $(K,w)$, called a \emph{weighted simplicial complex}.
Depending on a weight function $w$ (scalar products on cochain groups),  one can consider different versions of Laplace operators, i.e., $\mathcal{L}_{n}(K,w)$.
For example, if the weight function is constant and equal to $1$ on every simplex, then  the underlying Laplacian is the combinatorial Laplace operator denoted by $L_{n}(K)$, as analysed in \cite{DuvalShifted},\cite{Friedman}.\\
 If  $w$ satisfies the normalizing condition
\begin{equation}
\label{normalizing condition}
w(F)=\sum_{\substack{\bar{F}\in S_{n+1}(K):\\F\in \partial \bar{F} }} w(\bar{F}),
\end{equation} 
 for every $F\in S_{n}(K)$, which is not a facet of $K$, then $w$
determines the \emph{normalized Laplace operator} denoted by $\Delta_{n}$.
For more details on this topic the reader is invited to consult \cite{HorakNCL}.

Let $[\bar{F}]=[v_{0},\ldots, v_{n+1}]$ be an oriented  $(n+1)$-face of a complex $K$, such that $v_{0}<\ldots< v_{n+1}$.
Then \emph{the boundary of } $[\bar{F}]$ is
$$
\partial [\bar{F}]=\sum_{i}(-1)^{i}[v_{0},\ldots,\hat{v_{i}},\ldots,v_{n+1}].
$$
The faces $F_{i}=\{v_{0},\ldots,\hat{v_{i}},\ldots,v_{n+1}\}$ are part
of the boundary of
$\bar{F}$, 
and $\sgn([F_{i}],\partial [\bar{F}])=(-1)^{i}$, where $[F_{i}]=[v_{0},\ldots,\hat{v_{i}},\ldots,v_{n+1}]$.
By abuse of notation, we denote 
 the set of  all $n$-dimensional faces in the boundary of $\bar{F}$ by $\partial \bar{F}$.\\ 
 The up- and down-Laplace operators for an arbitrary weight function $w$ are 
 \small
$$
(\mathcal{L}_{n}^{up} f)([F])= \sum_{\substack{\bar{F}\in S_{n+1}(K):\\ F\in\partial \bar{F}}} \frac{ w(\bar{F})}{w(F)} f([F]) 
+   \sum_{\substack{F'\in S_{n}: F\neq F',\\ F,F'\in\partial  \bar{F}}} \frac{w(\bar{F})}{w(F)}\sgn([F],\partial [\bar{F}])\sgn([F'],\partial [\bar{F}])f([F']),
$$
\normalsize
and  
\small
$$
(\mathcal{L}^{down}_{n} f)([F])= \sum_{E\in \partial F}\frac{w(F)}{w(E)}f([F])+ \sum_{\substack{F'\in S_{n}(K):\\ F\cap F'=E}}\frac{w(F')}{w(E)}
 \sgn([E],\partial [F])\sgn([E],\partial [F']) f([F']).
$$
\normalsize

The combinatorial Laplace operator  for pairs (relative Laplacian) can   be defined in a similar manner, see \cite{DuvalRelative} for  its definition  on chain complexes.  The   relative  Laplacian on cochain complexes is dual to this definition.
In particular,  
let $K$ be a simplicial complex and $K_{0}$ its subcomplex.  Then  the
$n$-th cochain group of a pair 
$(K,K_{0})$ is $$C^{n}(K,K_{0};\rb):=\{f\in C^{n}(K,\rb)\mid f([F])=0, \textrm{ for every } F\in S_{n}(K_{0})\},$$ and 
 the coboundary operator
$\delta_{n}:C^{n}(K,\rb)\rightarrow C^{n+1}(K,\rb)$  induces a homomorphism of relative groups 
$\delta_{n}':C^{n}(K,K_{0};\rb)\rightarrow C^{n+1}(K,K_{0};\rb)$, i.e.
$\delta_{n}'(f)([\bar{F}])=f(\partial_{n}[\bar{F}])$, for $f\in C^{n}(K,K_{0};\rb)$. Moreover,
$H^{n}(K,K_{0};\rb)\cong \tilde{H}^{n}(K/K_{0};\rb)$. For details on
relative cohomology the reader can consult  \cite{Hatcher} or
\cite{Massey}. The Laplace operators of a pair are
$\mathcal{L}^{up}_{n}=\delta_{n}'^{*}\delta_{n}'$, $\mathcal{L}^{down}_{n}=\delta_{n-1}'\delta_{n-1}'^{*}$.
\begin{figure}[h!tp]
  \begin{center}
\includegraphics[scale=0.5]{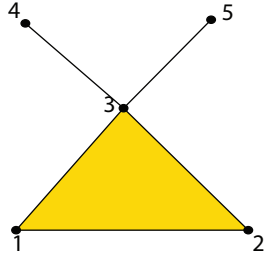}
  \end{center}
  \caption{Simplicial complex $K$}
  \label{figure relative }
\end{figure} 

\begin{exmp}
Let the weight function $w$ be constant and   $1$ on every face of 
$K$, as illustrated in Figure \ref{figure relative }.   
The up-Laplacians of the complex $K$ are  given by the following  matrices (with respect to lexicographical bases) 
\small
\[
\mathcal{L}^{up}_{0}(K) \!=\!\! \bordermatrix{~ & e_{[1]} & e_{[2]} & e_{[3]}& e_{[4]}& e_{[5]} \cr
                 e_{ [1]} & 2 & -1 & -1 & 0 & 0 \cr
                  e_{[2]} & -1 & 2 & -1 & 0 & 0\cr
                  e_{ [3]}&-1 & -1 & 4 & -1 & -1\cr
                  e_{ [4]}&0& 0 & -1 & 1 & 0\cr
                  e_{ [5]}& 0& 0 & -1 & 0 & 1\cr},
\quad
\mathcal{L}^{up}_{1}(K)\!= \!\! \bordermatrix{~ & e_{ [12]} & e_{ [13]} & e_{ [23]}& e_{ [34]}& e_{ [35]} \cr
                  e_{ [12]} & 1 & -1 & 1 & 0 & 0 \cr
                  e_{ [13]} & -1 & 1 & -1 & 0 & 0\cr
                  e_{ [23]}&1 & -1 & 1 & 0 & 0\cr
                  e_{ [34]}&0 & 0 & 0& 0 & 0\cr
                  e_{ [35]}& 0& 0 & 0 & 0 & 0\cr}.      
\]
\normalsize
The simplicial complex $K_{0}=\{\{1,2\}, \{3,4\}, \{1\}, \{2\},\{3\},\{4\}\}$  is a subcomplex of $K$, and the up-Laplace matrices of the pair $(K,K_{0})$ are
\small
\[
\mathcal{L}^{up}_{0}(K,K_{0}) \!=\!\! \bordermatrix{~ & e_{[5]}  \cr
                  e_{[5]}   & 1  \cr}
\quad
\textrm{and }
\quad
\mathcal{L}^{up}_{1}(K,K_{0})\!= \!\! \bordermatrix{~ & e_{[13]} & e_{[23]} & e_{[35]}] \cr
                  e_{[13]}  & 1 & -1 & 0  \cr
                  e_{[23]}  & -1 & 1 & 0 \cr
                  e_{[3 5]}  & 0 & 0 & 0\cr}.      
\]
\normalsize
Note that the matrix $\mathcal{L}^{up}_{1}(K,K_{0})$ is equal to the
matrix $\mathcal{L}^{up}_{1}(K)$ after deletion of the rows and columns indexed by 
$e_{[12]}$ and $ e_{[34]}$. Analogously for $\mathcal{L}^{up}_{0}(K,K_{0})$.
\end{exmp}

We shall not only consider positive valued weight functions, but also permit values equal to $0$, i.e.,  $w:\bigcup_{i}S_{i}(K)\rightarrow \rb^{+}\cup \{0\}$. In this case  we say that $(K,w)$  is a \emph{degenerate weighted simplicial complex}.  
The   above analysis  carries over to degenerate simplicial complexes, with the difference that $\delta_{n}^{*}$ is  called a formal adjoint in this case, and is defined as the standard  adjoint on its non-degenerate counterpart and extended by zero on simplices of weight zero, i.e.,
\begin{equation}
\label{formal adjoint}
\delta_{n}^{*}\e_{[\bar{F}]}=\sum_{\substack{F\in \partial \bar{F}:\\ w(F)\neq 0}}\frac{w(\bar{F})}{w(F)}\sgn([F],\partial [\bar{F}])e_{[F]},
\end{equation}
or equivalently, for $w(F)\neq 0$
\begin{equation}
\delta_{n}^{*}\bar{f}([F])=\sum_{\substack{ \bar{F}\in S_{n+1}:\\ F\in \partial \bar{F}, w(F)\neq 0}}\frac{w(\bar{F})}{w(F)}\sgn([F],\partial [\bar{F}])f([F]),
\end{equation}
and for $w(F)= 0$, $\delta_{n}^{*}\bar{f}([F])=0$.
\begin{rem}
Note that, in the case of degenerate simplicial complexes, a revised
version of the discrete Hodge theorem will hold. In particular, the
kernel of the $n$-Laplace operator will be isomorphic to  the direct
sum of  the  $n$-th homology group of a complex and the vector space over $\rb$ generated by  $n$-simplices of weight zero.
\end{rem}
\begin{rem}
Graphs with loops are  a special case of weighted graphs. For the graph Laplace operator, a loop effects only the degree of the corresponding 
vertex. The same is true for simplicial complexes.  We can define \emph{simplicial complexes with loops} in analogy to graphs with loops:
\begin{defn}
Let $[K]$ be an oriented simplicial complex.
If we allow oriented faces in which a vertex appears more than once, to be elements of $[K]$, then
we say that  $[v_{0},\ldots, v_{i}, v_{i+1}\ldots,
v_{i},v_{j+1},v_{n+1}]$ is a loop on the face $[v_{0},\ldots, v_{i}, v_{i+1}\ldots, \hat{v_{i}},v_{j+1},v_{n+1}]$
\end{defn}
It is not difficult to see that the boundary of a loop $[v_{0},\ldots, v_{i}, v_{i+1}\ldots, v_{i},v_{j+1},v_{n+1}]$ is zero. Thus, the Laplace operator defined on a simplicial complex $[K]$ with loops  will coincide with the one defined on $K$. The only place where we could account for loops is when defining the degree of a face, by adding the weight of the loop to its total degree.
Therefore, simplicial complexes with loops are already contained in the definition of weighted complexes, and we will not treat them separately.
We include simplicial complexes with loops here because they will arise naturally when we consider simplicial maps, which do not preserve dimensionality, see Section \ref{section 3}. 
\end{rem}
In this paper we develop a general framework for  eigenvalue
interlacing  of  generalized Laplace operators. This will, of course,
include the results  previously obtained on interlacing for combinatorial and normalized graph Laplacians, and adjacency matrices of  graphs.
Interlacing of the combinatorial Laplacian spectrum under deletion of an edge with non-pending vertices is a well-known result from algebraic graph theory, see \cite{Godsil}, Thm. 13.6.2.

Let $\lambda_{1}\leq  \ldots\leq \lambda_{n}$ be the eigenvalues of
$L^{up}_{0}(G)$ and  $\theta_{1}\leq \ldots\leq \theta_{n}$ the eigenvalues of $L^{up}_{0}(G-e)$, then
\begin{equation}
\label{Godsil interlacing}
\lambda_{k-1}\leq \theta_{k}\leq \lambda_{k}.
\end{equation}
A removal of a vertex and its incident edges was studied in \cite{Lotker} by Lotker, who proved the following
\begin{equation}
\label{Lotker interlacing}
\lambda_{k-1}\leq \theta_{k}\leq \lambda_{k+1}.
\end{equation}
Similar results were obtained by Chen et al. in \cite{Chen}, for the case  of the normalized graph Laplacians. In particular, 
\begin{equation}
\label{Chen interlacing}
\lambda_{k-1}\leq \theta_{k}\leq \lambda_{k+1},
\end{equation} 
 where  $\lambda_{i}$'s  and $\theta_{i}$'s are eigenvalues  ordered  non-decreasingly, of the normalized Laplacian of graph $G$ and of  $G-e$, respectively,  and more generally,
\begin{equation}
\label{Chen interlacing}
\lambda_{k-t}\leq \theta_{k}\leq \lambda_{k+t},
\end{equation}
where  the $\lambda_{i}$'s  and $\theta_{i}$'s are the eigenvalues
ordered  non-decreasingly, of the normalized Laplacian of the graph $G$ and  $G-H$, respectively, with $H$ being a spanning subgraph of $G$ on $t$ edges, and  $\lambda_{-t+1}=\ldots=\lambda_{-1}=\lambda_{0}=0$ and $\lambda_{n+1}=\ldots=\lambda_{n+t}=2$.
Butler in \cite{Butler} treats a  general case of interlacing on weighted graphs. 
He studies interlacing of $\Delta_{0}^{up}(G,w_{G})$ and $\Delta_{0}^{up}(L,w_{L})$, where $(L,w_{L})$ is  a weighted graph, obtained from 
$(G,w_{G})$ by deletion of its weighted subgraph  $(H,w_{H})$, which has  no isolated vertices.
The following inequalities are established
 \begin{equation}
 \label{Butler inequality}
\lambda_{k-t+1}\leq \theta_{k}\leq \left\{  \begin{array}{ll}
\lambda_{k+t-1} & \textrm{ if } H \textrm{ is bipartite,} \\
\lambda_{k+t} & \textrm{ otherwise, }
\end{array}
\right. 
\end{equation}
where $\lambda_{-t+1}=\ldots=\lambda_{-1}=0 $ and  $\lambda_{n}=\ldots=\lambda_{n+t-1}=2$, and $t$ is  the number of vertices of  $H$.
In Section \ref{section 2}  we analyze  interlacing of eigenvalues of
the generalised Laplacian $\mathcal{L}^{up}_{n}(K,w_{K})$ of a
weighted simplicial complex $(K,w_{K})$ and its weighted subcomplex
$(L,w_{L})$, and we obtain the following result. 
\begin{thm}
\label{main theorem}
Let $(K,w_{K})$ and $(H,w_{H})$  be  $(n+1)$-dimensional, weighted simplicial complexes such that 
their difference $(L,w_{L}):=(K-H,w_{K}-w_{L})$ is also a  simplicial
complex with non-negative weights.
Let $\lambda_{1}\leq \lambda_{2}\leq \ldots \leq \lambda_{N}$  and $\theta_{1}\leq \theta_{2}\leq \ldots \leq \theta_{N}$ 
 be the eigenvalues of $\mathcal{L}^{up}_{n}(K)$ and $\mathcal{L}^{up}_{n}(L)$, respectively. Then  we have
 $$\theta_{k}\leq \lambda_{ k+D_{H}} ,$$
for all $k = 1,2,\ldots, N-D_{H}$, and  for all $k = 1,2,\ldots, N$
 \begin{equation*}
 \lambda_{k-D_{\mathcal{W}}} \leq \theta_{k},
\end{equation*}
 where 
 $D_{\mathcal{W}}=\dim C^{n+1}( H,\rb)-\dim H^{n+1}( H,\rb) $, $D_{H}=\dim C^{n}(H,\rb)$, and  $0=\lambda_{1-D_{\mathcal{W}}}=\ldots=\lambda_{0} $.
\end{thm}
Compared to  (\ref{Butler inequality}) and to (\ref{Chen
  interlacing}), a special case of the theorem above will yield a 
  sharper  lower interlacing inequality for the graph Laplacian with normalized weights, see Corollary \ref{Corollary graph interlacing}, and
  inequalities on the combinatorial graph Laplacian are a special case of Theorem  \ref{Proposition interlacing}. In the same section we obtain many well-known eigenvalue bounds   as a special case of more general inequalities, which are consequences of Theorem \ref{main theorem} and Theorem \ref{Proposition interlacing}.

In Section \ref{section 3} we analyze coverings of simplicial complexes, point out the failure of the definition of combinatorial covering used so far
\cite{Rotman}, \cite{Gustavson}, give the correct discretization of covering spaces and maps from the continious setting,   and  state a correct proof (Corollary \ref{Corollary strong coverings}), of  Theorem 4.4. from \cite{Gustavson}.
Furthermore, we analyze the general case when $\varphi:(K,w_{K})\rightarrow (K',w_{K'})$ is a simplicial map, and prove interlacing theorems for two natural choices of  the weight function $w_{K'}$ i.e., Theorem \ref{theorem intercing simplicial map rule(ii)} and Theorem \ref{Theorem interlacing 2nd choice of weights}.
In Section \ref{section 4} we study special cases of simplicial maps,  contractions and collapses,  and prove interlacing inequalities for the eigenvalues of the combinatorial and normalized Laplacians of such complexes.  As a special case, we prove the interlacing inequalities on graphs under the operation of contraction of an edge.
Section \ref{section 5}  treats the case of eigenvalues of the relative Laplacian.

\section{Interlacing theorems: deletion of a subcomplex}
\label{section 2}
As shown in \cite{HorakNCL}, the generalized  combinatorial Laplace operator $\mathcal{L}_{n}$ 
is self-adjoint positive definite and the set of its non-zero eigenvalues   is a disjoint union of the non-zero eigenvalues  of $\mathcal{L}^{up}_{n}$ and  $\mathcal{L}^{down}_{n}$.
Thus, it suffices to investigate only one of the following families of spectra
$$\{\s(\mathcal{L}_{n}(K))\mid -1\leq n \leq d\}\textrm{,}\; \{\s(\mathcal{L}_{n}^{up}(K))\mid -1\leq n \leq d\} \;\textrm{or}\; \{\s(\mathcal{L}_{n}^{down}(K))\mid 0\leq n\leq d\},$$
where $\s$ denotes the multiset of  eigenvalues. The operators
$\mathcal{L}^{up}_{n}$ and $\mathcal{L}^{down}_{n}$ are self-adjoint
and compact, whence we can apply the 
Spectral theorem,  the Variational characterization theorem, and the Min-max theorem (also known as Courant-Fischer-Weyl min-max principle) 
to characterize their eigenvalues.
We state these theorems below.  For proofs, see \cite{Stadler2007}.
\begin{thm}[Spectral theorem] 
\label{spectral theorem} 
Let A be a 
compact self-adjoint operator on a Hilbert space V. Then, there is an orthonormal basis of V consisting of eigenvectors of A.
\end{thm}
\begin{thm}[Variational characterization theorem]
\label{variational theorem}
 Let $f_{1},\ldots,f_{m}$ denote orthogonal  eigenfunctions corresponding
 to  eigenvalues $\lambda_{1}\leq \lambda_{1}\leq \ldots\leq \lambda_{m}$ of  a compact, self-adjoint operator $A$ on a Hilbert space $V$.
Let $F_{i}=\{f_{1},\ldots,f_{i}\}$ be the set of the first $i$ eigenfunctions  of $A$, and $F_{i}^{\perp}$ its orthogonal complement.
Then,
\begin{displaymath}
 \lambda_{i}=\min_{g\in F_{i-1}^{\perp} } \frac{(g,A g)}{(g,g)}=\max_{g\in F_{i}^{\perp} } \frac{(g,A g)}{(g,g)}.
\end{displaymath}
\end{thm}
The quantity  $ \frac{(g,A g)}{(g,g)}$ is commonly 
 called the \textit{Rayleigh quotient} $\mathcal{R}_{A}(g)$. 
\begin{thm}[Min-max theorem]
\label{min-max theorem}
Let $\mathcal{V}_{k}$  denote a $k$-dimensional subspace  of $V$. Then,
\begin{equation}
\label{min-max equation}
\lambda_{k}=\min_{\mathcal{V}_{k}}\max_{g\in V_{k}}\mathcal{R}_{A}(g)=\max_{\mathcal{V}_{m-k+1}}\min_{g\in \mathcal{V}_{m-k+1}}\mathcal{R}_{A}(g)
\end{equation}
and
\begin{equation}
\label{min-max equation1}
\lambda_{k}=\min_{\mathcal{V}_{n-k}}\max_{g\perp V_{n-k}}\mathcal{R}_{A}(g)=\max_{\mathcal{V}_{k-1}}\min_{g\perp \mathcal{V}_{k-1}}\mathcal{R}_{A}(g).
\end{equation}
\end{thm}
We shall use the following slight modification of the min-max theorem.
\begin{lem}
\label{lemma minmax}
Let $\mathcal{V}_{k}$ be, as before, a $k$-dimensional subspace of a Hilbert  space $V$ and let $A:V\rightarrow V$ be a compact, self-adjoint  operator. Then 
$$\lambda_{k-s}=\min_{V_{k}}\max_{x\in V_{k-s}}A(x)=  \min_{V_{k-s}}\max_{x\in V_{k-s}}A(x)=\lambda_{k-s}.$$
Furthermore,
\[
\lambda_{k}=\min_{V_{k}}\max_{x\in V_{k}}A(x)\geq \min_{V_{k-s}}\max_{x\in V_{k-s}}A(x)= \min_{V_{k}}\max_{x\in V_{k-s}}A(x)=\lambda_{k-s}
\]
\end{lem}
\begin{rem}
\label{remark convention}
 Theorem \ref{spectral theorem}, Theorem \ref{variational theorem},
 Theorem \ref{min-max theorem},  and Lemma \ref{lemma minmax}    also
 hold  when the vector space $V$ is  equipped with a degenerate inner product, with the convention that $\mathcal{R}_{A}(g)=0$, if 
 $(g,g)=0$.
\end{rem}
From the previous considerations we deduce the following. Given a weighted simplicial complex $(K,w_{K})$, 
there exists  an orthonormal basis of  $C^{n}(K,\rb)$ (with
  the scalar product  induced by the weight function as discussed in Section \ref{section 1})  that consists of eigenfunctions $f_{1},\ldots,f_{m}$ corresponding
 to  eigenvalues $\lambda_{1}\leq \ldots\leq\lambda_{m}$ of $\mathcal{L}^{up}_{n}(K)$. Moreover, for every $g\in C^{n}(K,\rb)$,
\begin{displaymath}
 \mathcal{L}_{n}^{up}(g)=\sum_{i=1}^{m}\lambda_{i}(g,f_{i})f_{i}
\end{displaymath}
and
\begin{displaymath}
 (g,\mathcal{L}_{n}^{up}(g)) = \sum_{i=1}^{m}\lambda_{i}(g,f_{i})^{2}.
\end{displaymath}
\begin{defn}
A weighted simplicial complex $(H,w_{H})$ is a \emph{subcomplex} of $(K,w_{K})$ iff 
 $H$ is a subcomplex of $K$ and $w_{H}(F)\leq w_{K}(F)$ for every face $F$ of $H$.
\end{defn} 
 \begin{defn}
 \label{Difference}
 A \emph{proper difference} of the weighted simplicial complexes  $(K,w_{K})$ and $(H,w_{H})$ is a degenerate weighted simplicial complex $(L,w_{L})$, such that 
$L\equiv K$,  $w_{L}:=w_{K}-w_{H}$, and $L'=\{F\in K\mid w_{L}(F)>0\}$ is a simplicial complex.
 \end{defn}
 \begin{rem}
\label{Remark H}
When analyzing $\mathcal{L}_{n}^{up}(K)$, it suffices to consider the  $(n+1)$-skeleton of a simplicial complex $K$.
For details see \cite{HorakNCL}, Section 6.1. Therefore, without  loss of generality, we can assume that $K$ and $H$ are  $(n+1)$-dimensional.
\end{rem}
 \begin{prop}
 \label{Prop commute}
 Let $(K,w_{K})$ be a weighted simplicial complex, $(H,w_{H})$ its subcomplex, and
 $(L,w_{L})$ their proper difference. 
 Let $i_{n}:C^{n}(L,\rb)\rightarrow C^{n}(K,\rb)$ be an inclusion  map, such that  $(i_{n}g)[F]=g([F])$, for every $F\in S_{n}(K)$, $g\in C^{n}(L,\rb)$ and 
 let  $\pi_{n+1}:C^{n+1}(K,\rb)\rightarrow C^{n+1}(L,\rb)$ be a
 projection map such that   $(\pi_{n+1} \bar{f})[\bar{G}]=
 \bar{f}([\bar{G}])$, for all $\bar{G}\in S_{n+1}(L)$ and $\bar{f}\in
 C^{n+1}(K,\rb)$. We consider the following formal\footnote{The
   epithete "formal"  emphasizes  that the scalar product may be
   degenerate.} adjoints $i_{n}^{*}: C^{n}(K,\rb)\rightarrow
 C^{n}(L,\rb)$ and   $\pi_{n+1}^{*}:C^{n+1}(L,\rb)\rightarrow
 C^{n+1}(K,\rb)$ of the inclusion and projection maps, respectively, 
 \begin{equation}
i^{*}_{n}f([G])=\left\{\begin{array}{ll}
\frac{w_{K}(G)}{w_{L}(G)}f(i_{n}[G])=\frac{w_{K}(G)}{w_{L}(G)}f(i_{n}[G]) & \textrm{ if } w_{L}(G)>0,\\
0 &  \textrm{ otherwise,} \\
\end{array}
\right. 
\end{equation}
and
\begin{equation}
\pi_{n+1}^{*}\bar{g}([\bar{F}])=\frac{w_{L}(\bar{F})}{w_{K}(\bar{F})}g(\pi_{n+1}[\bar{F}])=\frac{w_{L}(\bar{F})}{w_{K}(\bar{F})}g([\bar{F}]).
\end{equation}
Then for every $0\leq n < \dim K$  the following diagrams commute.
\begin{equation}
\label{two commutative diagrams}
\begin{CD}
C^{n+1}(K,\rb)  @<\delta_{K}<<   C^{n}(K,\rb) \\
 @VV\pi_{n+1} V    @AAi_{n} A \\
C^{n+1}(L,\rb)  @<\delta_{L}<<   C^{n}(L,\rb) \\
\end{CD}
\qquad
\begin{CD}
C^{n+1}(K,\rb)  @>\delta^{*}_{K}>>   C^{n}(K,\rb) \\
 @AA\pi^{*}_{n+1} A    @VVi^{*}_{n} V \\
C^{n+1}(L,\rb)  @>\delta^{*}_{L}>>   C^{n}(L,\rb) \\
\end{CD}
\end{equation}
\end{prop}
\begin{proof}
The proof is elementary.
\end{proof}
For brevity, in what follows, we  omit the index $n$ in $i_{n}$, $\pi_{n}$, $i^{*}_{n}$, and $\pi^{*}_{n}$.
From the definitions of inclusion and projection maps, it is straightforward to calculate
\begin{equation}
\label{projection adjoint}
\pi^{*}\pi \bar{f}([\bar{F}])=\frac{w_{L}(\bar{F})}{w_{K}(\bar{F})}\bar{f}([\bar{F}]),
\end{equation}
and
\begin{equation}
\label{inclusion adjoint}
i^{*}i g([G])=\left\{\begin{array}{ll}
\frac{w_{K}(G)}{w_{L}(G)}g([G]) & \textrm{ for } w_{L}(G)>0,\\
 0&  \textrm{ otherwise. }\\
\end{array}
\right.
\end{equation}
Furthermore, from the commutativity of the diagrams  (\ref{two commutative diagrams}), we derive
  \begin{align}
  \mathcal{R}_{\mathcal{L}_{n}^{up}(L)}(g)=  \frac{(\delta^{*}_{L} \delta_{L}g,g)}{(g,g)}&=   \frac{(i^{*}\delta^{*}_{K}\pi^{*}\pi\delta_{K}ig,g)}{(g,g)}\notag\\
   &=    \frac{(\delta^{*}_{K}\pi^{*}\pi\delta_{K}ig,ig)}{(g,g)}\notag\\
   &=    \frac{(\delta^{*}_{K}\pi^{*}\pi\delta_{K}ig,ig)}{(ig,ig)}\frac{(ig,ig)}{(g,g)}\notag\\
    &=    \frac{(\pi^{*}\pi\delta_{K}ig,\delta_{K}ig)}{(ig,ig)}\frac{(i^{*}ig,g)}{(g,g)}\notag\\
     &=    \frac{(\pi^{*}\pi\delta_{K}ig,\delta_{K}ig)}{(\delta_{K}ig,\delta_{K}ig)}
     \frac{(\delta_{K}ig,\delta_{K}ig)}{(ig,ig)}\frac{(ig,ig)}{(g,g)}\notag\\
    &=  \mathcal{R}_{\pi^{*}\pi}(\delta_{K}ig)   \mathcal{R}_{\mathcal{L}_{n}^{up}(K)}(ig)\mathcal{R}_{i^{*}i}(g)\label{decomposition}.
 \end{align}
 Here we adopt the convention of Remark \ref{remark convention}  that $\mathcal{R}_{A}(f)=0$ whenever 
 the denominator of the Rayleigh quotient is $0$. One can easily check that this convention does not affect the equality above.
 \begin{rem}
 We point out that with the simplicial complex $L'$ that avoids the
 degeneracy of the scalar product, instead of $L$, we cannot obtain a good lower bound on 
 $\theta_{k}$ because we cannot precisely calculate  the dimension of
 $C^{n}(L\cap H,\rb)/ (C^{n}(L\cap H,\rb)\cap \ker \delta_{H})$, which
 is necessary to establish  the equivalence (\ref{lemma projection equivalence}).
 \end{rem}
 The following lemmas characterize the operator $\pi^{*}\pi$.
 \begin{lem}
 \label{lemma projection inequality}
The eigenvalues of the operator $\pi^{*}\pi: C^{n+1}(K,\rb)\rightarrow C^{n+1}(K,\rb)$ are $\{\frac{w_{L}(\bar{F})}{w_{K}(\bar{F})}\mid  \bar{F}\in S_{n+1}(K)\}$,  hence
  $\mathcal{R}_{\pi^{*}\pi}(f)\leq  1$.
 \end{lem}
 \begin{proof}
 The proof is a direct consequence of (\ref{projection adjoint}).
 \end{proof}
 \begin{lem}
 \label{lemma projection}
 The following equivalence holds
 \begin{equation}
 \label{lemma projection equivalence}
 \mathcal{R}_{\pi^{*}\pi}(\delta_{K}ig)=1 \textrm{ iff } g_{\lvert C^{n}(H,\rb)}\in \ker \delta_{H}.
 \end{equation}
The set of functions $g\in C^{n}(L,\rb)$ satisfying (\ref{lemma projection equivalence}) is a vector space orthogonal to 
  $\mathcal{W}:= \coker \delta_{H}$, and the dimension of $\mathcal{W}$
  is  $D_{\mathcal{W}}= \dim C^{n+1}(H,\rb)-\dim H^{n+1}( H,\rb)$.
 \end{lem}
    \begin{proof}
Let $\bar{f}:=\delta_{K}ig$, then according to (\ref{projection adjoint}) we get
\begin{equation}
\label{eq1}
\mathcal{R}_{\pi^{*}\pi}(\bar{f})=\frac{\sum_{\bar{F}\in S_{n+1}(L)}w_{L}(\bar{F})\bar{f}([\bar{F}])^{2} }{\sum_{\bar{F}\in S_{n+1}(K)}w_{K}(\bar{F})\bar{f}([\bar{F}])^{2}}.
\end{equation}
Thus, (\ref{eq1}) is equal to $1$ iff  
\begin{equation}
\label{equation projection 1}
w_{K}(\bar{F})\bar{f}([\bar{F}])^{2}=(w_{K}(\bar{F})-w_{H}(\bar{F}))\bar{f}([\bar{F}])^{2},
\end{equation} 
for every  $\bar{F}\in S_{n+1}(K)$.
The relation (\ref{equation projection 1}) holds if 
$\bar{f}$ is identically equal to zero on the subcomplex $H$, i.e., if the restriction of $g\in C^{n}(L,\rb)$ on $C^{n}(H,\rb)$ is in the kernel of 
 $\delta_{H}$.
Therefore,  the function $g$ is orthogonal to   $\mathcal{W}:= \coker \delta_{H}$.
For the dimension of $\mathcal{W}$, we have the following short exact
sequences that split, 
$0\rightarrow \ker \delta_{n}\rightarrow C^{n}\rightarrow \im \delta_{n}\rightarrow 0$, $0\rightarrow \im \delta_{n-1}\rightarrow  \ker \delta_{n} \rightarrow H^{n}\rightarrow 0$, and deduce
\begin{equation}
\dim \coker \delta_{n}=\dim\im\delta_{n}=\dim \ker \delta_{n+1}-\dim H^{n+1}.
\end{equation}
According to Remark \ref{Remark H},  $H$ and $K$ are  $(n+1)$-dimensional, therefore 
$\ker \delta_{n+1}=C^{n+1}(H,\rb)$ and
\begin{equation}
\dim\im\delta_{H}= \dim C^{n+1}( H,\rb)-\dim H^{n+1}( H,\rb).
\end{equation}
The other direction in  the equivalence (\ref{lemma projection equivalence}) is trivial.
\end{proof}
The eigenvalues of the operator $i^{*}i$ are characterized by the following lemmas.
 \begin{lem}
 \label{lemma inclusion}
Let $g\in C^{n}(L,\rb)$. Then, the following implications are true
\begin{itemize}
\item [(i)]$(g,g)>0\Rightarrow\mathcal{R}_{i^{*}i}(g)\geq 1$.
\item [(ii)]$(g,g)=0\Rightarrow\mathcal{R}_{\pi^{*}\pi}(\delta_{K}ig)=0$.
\end{itemize}
 \end{lem}
 \begin{proof}
 From  (\ref{inclusion adjoint}), it  follows that
 \begin{equation}
 \label{Rayleigh inclusion}
\frac{\sum_{G\in S_{n}(L)}w_{K}(G)g([G])^{2}}{\sum_{G\in S_{n}(L)}w_{L}(G)g([G])^{2}}.
\end{equation}
 Since
 $w_{K}(G)\geq w_{L}(G)$, then $(i)$ holds.
 Implication $(ii)$ is a direct consequence of  (\ref{eq1}).
 \end{proof}

  \begin{lem}
  \label{lemma inclusion supprot}
  Let $g\in C^{n}(L,\rb)$, then 
  the following equivalence holds.
  \begin{equation}
  \label{eq2}
  \mathcal{R}_{i^{*}i}(g)=1 \textrm{ iff } g\perp C^{n}(H,\rb)
  \end{equation}
  \end{lem}
\begin{proof}
Due to (\ref{Rayleigh inclusion}), we have  $\mathcal{R}_{i^{*}i}(g)=1$ iff $g([G])=0$ for every $G\in S_{n}(H)$.
\end{proof}
\begin{proof}[Proof of Theorem \ref{main theorem}]
Denote the dimension of the $n$-th cochain group  $C^{n}(K,\rb) $ by $N$ and  
let $\lambda_{1}\leq \lambda_{2}\leq\ldots \leq \lambda_{N}$ and  $\theta_{1}\leq \theta_{2}\leq\ldots \leq \theta_{N}$ be the  eigenvalues of the operators $\mathcal{L}^{up}_{n}(K)$ and  $\mathcal{L}^{up}_{n}(L)$,  respectively. Then, 
\begin{align}
\theta_{k}=&\min_{\mathcal{V}_{k}}\max_{g\in \mathcal{V}_{k}}  \mathcal{R}_{\mathcal{L}_{n}^{up}(L)}(g)\notag\\
=&\min_{\mathcal{V}_{k}}\max_{g\in \mathcal{V}_{k}}  \mathcal{R}_{\pi^{*}\pi}(\delta_{K}ig)   \mathcal{R}_{\mathcal{L}_{n}^{up}(K)}(ig)
     \mathcal{R}_{i^{*}i}(g)\label{eq01a}\\
\geq &\min_{\mathcal{V}_{k}}\max_{g\in \mathcal{V}_{k}}  \mathcal{R}_{\pi^{*}\pi}(\delta_{K}ig)   \mathcal{R}_{\mathcal{L}_{n}^{up}(K)}(ig)\label{ineq01}\\
\geq &\min_{\mathcal{V}_{k}}\max_{g\in \mathcal{V}_{k},g\perp \mathcal{W}} \mathcal{R}_{\mathcal{L}_{n}^{up}(K)}(ig)\label{ineq02}\\
\geq & \min_{\mathcal{V}_{k}}\max_{g\in \mathcal{V}_{k},g\perp \mathcal{V}_{D_{\mathcal{W}}}} \mathcal{R}_{\mathcal{L}_{n}^{up}(K)}(ig)\label{ineq03}.
\end{align}
(\ref{eq01a}) comes from (\ref{decomposition}). (\ref{ineq01})  is a
consequence of Lemma \ref{lemma inclusion}. (\ref{ineq02})
follows from Lemma \ref{lemma projection} and the fact that we are
taking the maximum over a smaller set, whereas 
(\ref{ineq03})  follows since we are performing a wider minimization than in (\ref{ineq02}).
Due to  Lemma \ref{lemma minmax}, from   (\ref{ineq03}) we obtain
\begin{align}
\theta_{k}&\geq \min_{\mathcal{V}_{k}}\max_{g\in \mathcal{V}_{k-D_{\mathcal{W}}}} \mathcal{R}_{\mathcal{L}_{n}^{up}(K)}(ig)\notag\\ 
&=   \min_{\mathcal{V}_{k-D_{\mathcal{W}}}}\max_{g\in \mathcal{V}_{k-D_{\mathcal{W}}}} \mathcal{R}_{\mathcal{L}_{n}^{up}(K)}(ig)=\lambda_{k-D_{\mathcal{W}}}\notag .
\end{align}
Therefore,
\begin{equation}
\label{lower inequality}
\theta_{k}\geq \lambda_{k-D_{\mathcal{W}}}.
\end{equation}
We derive an upper interlacing inequality  for  $\theta_{k}$  from (\ref{decomposition}) and Theorem \ref{min-max theorem}, as well, i.e.,
\begin{align}
\theta_{k}&=\max_{\mathcal{V}_{N-k+1}}\min_{f\in \mathcal{V}_{N-k+1}}  \mathcal{R}_{\mathcal{L}_{n}^{up}(L)}(f)\notag\\
&=\max_{\mathcal{V}_{N-k+1}}\min_{f\in \mathcal{V}_{N-k+1}}  \mathcal{R}_{\pi^{*}\pi}(\delta_{K}if)   \mathcal{R}_{\mathcal{L}_{n}^{up}(K)}(if)
     \mathcal{R}_{i^{*}i}(f).\notag
 \end{align}
 Since      $ \mathcal{R}_{\pi^{*}\pi}(\delta_{K}ig) \leq 1$ , we have
\begin{align}
\theta_{k} \leq & \max_{\mathcal{V}_{N-k+1}}\min_{g\in \mathcal{V}_{N-k+1}}   \mathcal{R}_{\Delta_{K}}(ig) \mathcal{R}_{i^{*}i}(g) \label{ineqZX}\\
 \leq & \max_{\mathcal{V}_{N-k+1}}\min_{g\in \mathcal{V}_{N-k+1}, g\perp C^{n}(H,\rb)}   \mathcal{R}_{\Delta_{K}}(ig)\label{ineq04Z}\\
 =& \max_{\mathcal{V}_{N-k+1}}\min_{g\in \mathcal{V}_{N-k+1-D_{H}}}   \mathcal{R}_{\Delta_{K}}(ig)\label{ineq05}\\
=& \max_{\mathcal{V}_{N-k+1-D_{H}}}\min_{g\in \mathcal{V}_{N-k+1-D_{H}}}   \mathcal{R}_{\Delta_{K}}(ig)\label{ineq06}\\
=&\lambda_{k+D_{H}}.
\end{align} 
(\ref{ineq04Z}) follows from Lemmas \ref{lemma inclusion} (i) and  \ref{lemma inclusion supprot} and the fact that
we are minimizing over a subset of the set $\mathcal{V}_{N-k+1}$
occuring in (\ref{ineqZX}), whereas     (\ref{ineq06}) is a consequence of Lemma \ref{lemma minmax}.\\
Together with (\ref{lower inequality}) this proves Theorem \ref{main theorem}.
\end{proof}
\begin{rem}
Other than the requirement that the proper difference of the complexes $K$ and $H$ must exist,  no restrictions  on weight functions nor on simplicial complexes are imposed. Thus, the interlacing theorem Theorem \ref{main theorem} holds for the generalized Laplace operator.
Note that, unlike in \cite{Butler},  we allow $H$ to contain isolated $n$-simplices (isolated vertices in the case of graphs).
\end{rem}

The spectrum of the combinatorial Laplace operator $L^{up}_{i}(K)$ is bounded from above by the number of vertices of the complex $K$ (for details, see Proposition 4.2. \cite{DuvalShifted});  hence, we have the following version of Theorem \ref{main theorem}.
\begin{thm}
\label{main theorem combinatorial Laplacian}
Let $(K,w_{K}\equiv 1)$  be an  $(n+1)$-dimensional simplicial complex on a vertex set of cardinality $\lvert V\lvert$, and let $(H,w_{H}\equiv 1)$  be its $(n+1)$-subcomplex, such that 
their proper difference $(L,w_{L}):=(K-H,w_{K}-w_{L})$ exists.
Let $\lambda_{1}\leq \lambda_{2}\leq \ldots \leq \lambda_{N}$  and $\theta_{1}\leq \theta_{2}\leq \ldots \leq \theta_{N}$ 
 be the eigenvalues of $L^{up}_{n}(K)$ and $L^{up}_{n}(L)$ respectively. Then,  for all $k = 1,2,\ldots, N$, we have
 $$\lambda_{k-D_{\mathcal{W}}}\leq \theta_{k}\leq \lambda_{k+D_{H}} ,$$
 where 
 $D_{\mathcal{W}}=\dim C^{n+1}( H,\rb)-\dim H^{n+1}( H,\rb)$,  $0=\lambda_{1-D_{\mathcal{W}}}=\ldots=\lambda_{0} $,
 and $\lvert V\lvert =\lambda_{N}+1=\ldots=\lambda_{N}+D_{H}$.
\end{thm}

The eigenvalues of  the normalized combinatorial Laplacian $\Delta_{n}^{up}$ are bounded by $n+2$ (see \cite{HorakNCL}). Hence, we have the following theorem.
\begin{thm}
\label{main theorem normalized Laplacian}
Let $(K,w_{K})$ and $(H,w_{H})$  be  $(n+1)$-dimensional, weighted simplicial complexes such that 
their proper difference $(L,w_{L}):=(K-H,w_{K}-w_{L})$ exists.
Let $\lambda_{1}\leq \lambda_{2}\leq \ldots \leq \lambda_{N}$  and $\theta_{1}\leq \theta_{2}\leq \ldots \leq \theta_{N}$ 
 be the eigenvalues of $\Delta^{up}_{n}(K)$ and $\Delta^{up}_{n}(L)$, respectively. Then, for all $k = 1,2,\ldots, N$, we have
 $$\lambda_{k-D_{\mathcal{W}}}\leq \theta_{k}\leq \lambda_{k+D_{H}} ,$$
 where 
 $D_{\mathcal{W}}=\dim C^{n+1}( H,\rb)-\dim H^{n+1}( H,\rb)$,  $0=\lambda_{1-D_{\mathcal{W}}}=\ldots=\lambda_{0} $,
 and $n+2=\lambda_{N}+1=\ldots=\lambda_{N+D_{H}}$.
\end{thm}
As a special case of Theorem \ref{main theorem normalized Laplacian} we derive the following corollary on interlacing  for the normalized graph Laplacian.
\begin{coll}
\label{Corollary graph interlacing}
Let $(K,w_{K})$ and $(H,w_{H})$  be  weighted  graphs, and let $(L,w_{L}):=(K-H,w_{K}-w_{L})$ be their proper difference.
Let $\lambda_{1}\leq \lambda_{2}\leq \ldots \leq \lambda_{N}$  and $\theta_{1}\leq \theta_{2}\leq \ldots \leq \theta_{N}$ 
 be the eigenvalues of $\Delta^{up}_{0}(K)$ and $\Delta^{up}_{0}(L)$ respectively, then for all $k = 1,2,\ldots, N$ we have
 \begin{equation}
 \label{Horak interlacing}
 \lambda_{k-\dim C^{0}(H,\rb)+\dim H^{0}(H,\rb) }\leq \theta_{k}\leq \lambda_{k+\dim C^{0}(H,\rb)} ,
 \end{equation}
 where  $0=\lambda_{1- \dim C^{0}(H,\rb)+\dim H^{0}(H,\rb) }=\ldots=\lambda_{0} $, and  $2=\lambda_{N+1}=\ldots=\lambda_{N+\dim C^{0}(H,\rb)}$.
\end{coll}
\begin{proof}
From the formula for the Euler characteristic in terms of Betti numbers,   $\chi=\sum_{i}(-1)^{i}\dim C^{i}(K,\rb)=\sum_{i}(-1)^{i}\dim H^{i}(K,\rb)$, we get
 $$\dim C^{1}(G,\rb)-\dim H^{1}(G,\rb)=\dim C^{0}(G,\rb)-\dim H^{0}(G,\rb),$$ for every graph $G$. 
 Inserting this in the expression for $D_{\mathcal{W}}$, where  $D_{\mathcal{W}}=\dim C^{1}(H,\rb)-\dim H^{1}(H,\rb)$, we obtain the desired formula.
\end{proof}
\begin{rem}
A similar claim holds for the combinatorial graph Laplacian with a slight modification. Instead of $2=\lambda_{N+1}=\ldots=\lambda_{N+\dim C^{0}(H,\rb)}$, we take $ \dim C^{0}(K,\rb)=\lambda_{N+1}=\ldots=\lambda_{N+\dim C^{0}(H,\rb)}$.
\end{rem}
\begin{rem}
The interlacing inequalities of Theorem \ref{main theorem}  for non-bipartite graphs $H$ are at least as good as the interlacing inequalities in  (\ref{Butler inequality}). Moreover, for a general disconnected graph $H$,  we obtain a finer lower bound than Butler \cite{Butler}.
In addition,  (\ref{Horak interlacing}) includes even the cases when $H$ contains isolated vertices.
When the graph $H$ is bipartite, the upper interlacing inequality in (\ref{Butler inequality}) is better than in Theorem \ref{main theorem}; otherwise, the two inequalities are the same.
However,  the estimate of Butler for bipartite graphs is tightly
connected to the normalization properties of the normalized graph Laplacian, thus it  cannot be generalized to other graph Laplacians.
\end{rem}
\begin{exmp}
Let $(K,w_{K})$ and $(L,w_{L})$ be the simplicial complexes shown in Figure \ref{ig01} and Figure \ref{ig02}, respectively, such that the weight  functions $w_{K}$ and $w_{L}$ take value $1$ on all edges.
The eigenvalues of the normalized  graph Laplacian
$\Delta^{up}_{0}(K)$  are  $0, 0.73,1,1,  1.42,$ and $1.85$, and the eigenvalues of 
 $\Delta^{up}_{0}(L)$ are $0, 0.19, 0.89, 1.5,1.5,$ and $1.92 $.
The lower interlacing inequality from Corollary \ref{Corollary graph interlacing} yields
 $\lambda_{k-3}\leq \theta_{k}$, while the inequality (\ref{Butler inequality}) gives
 $\lambda_{k-5}\leq \theta_{k}$.
\end{exmp}
  \begin{figure}[h!tp]
\label{figure interlacing graph}
  \begin{center}
 \subfigure[raggedright,scriptsize][$K$ ]{\label{ig01}\includegraphics[scale=0.4]{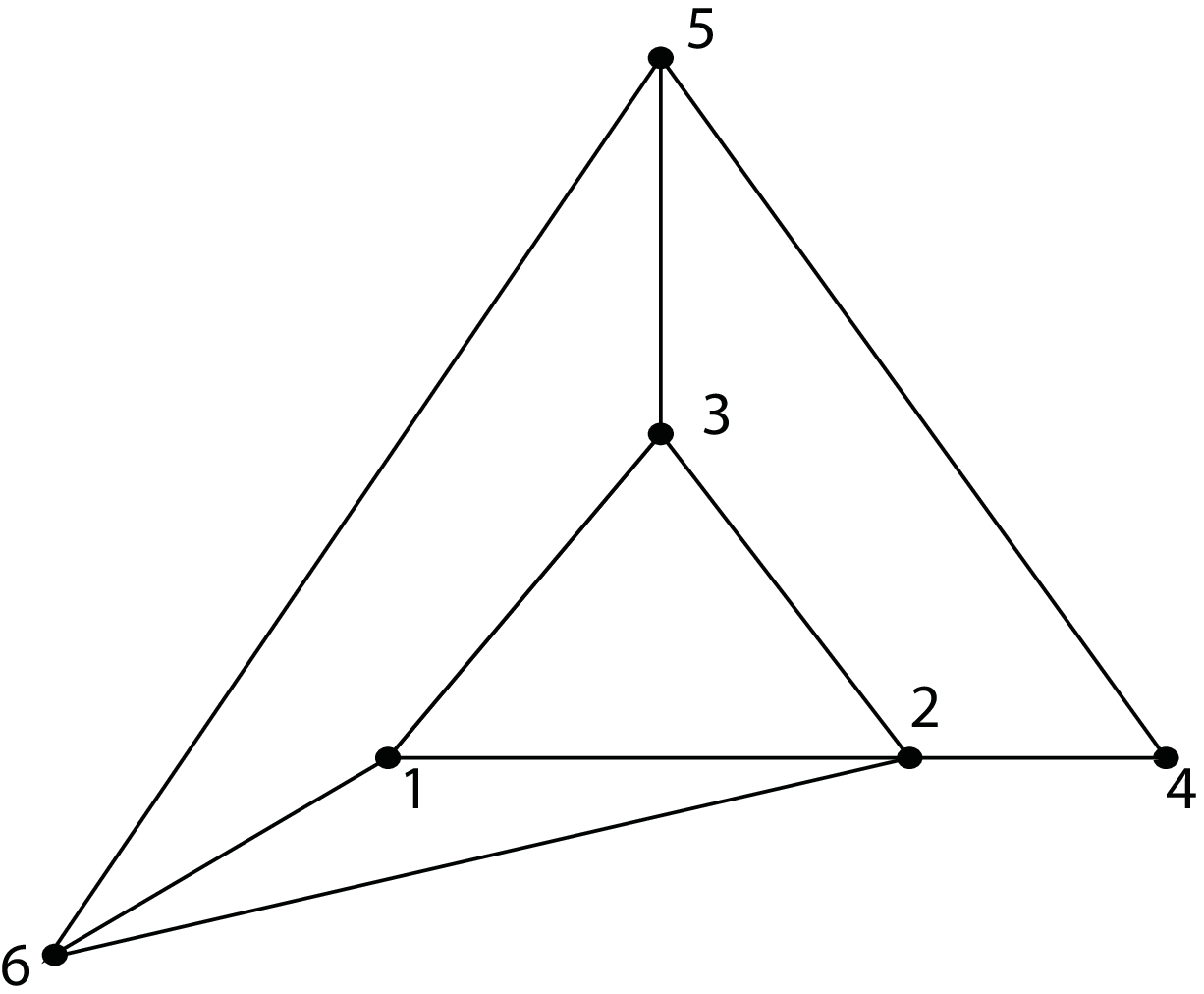}} \qquad\qquad\qquad
    \subfigure[raggedright,scriptsize][ $L$ ]{\label{ig02}\includegraphics[scale=0.4]{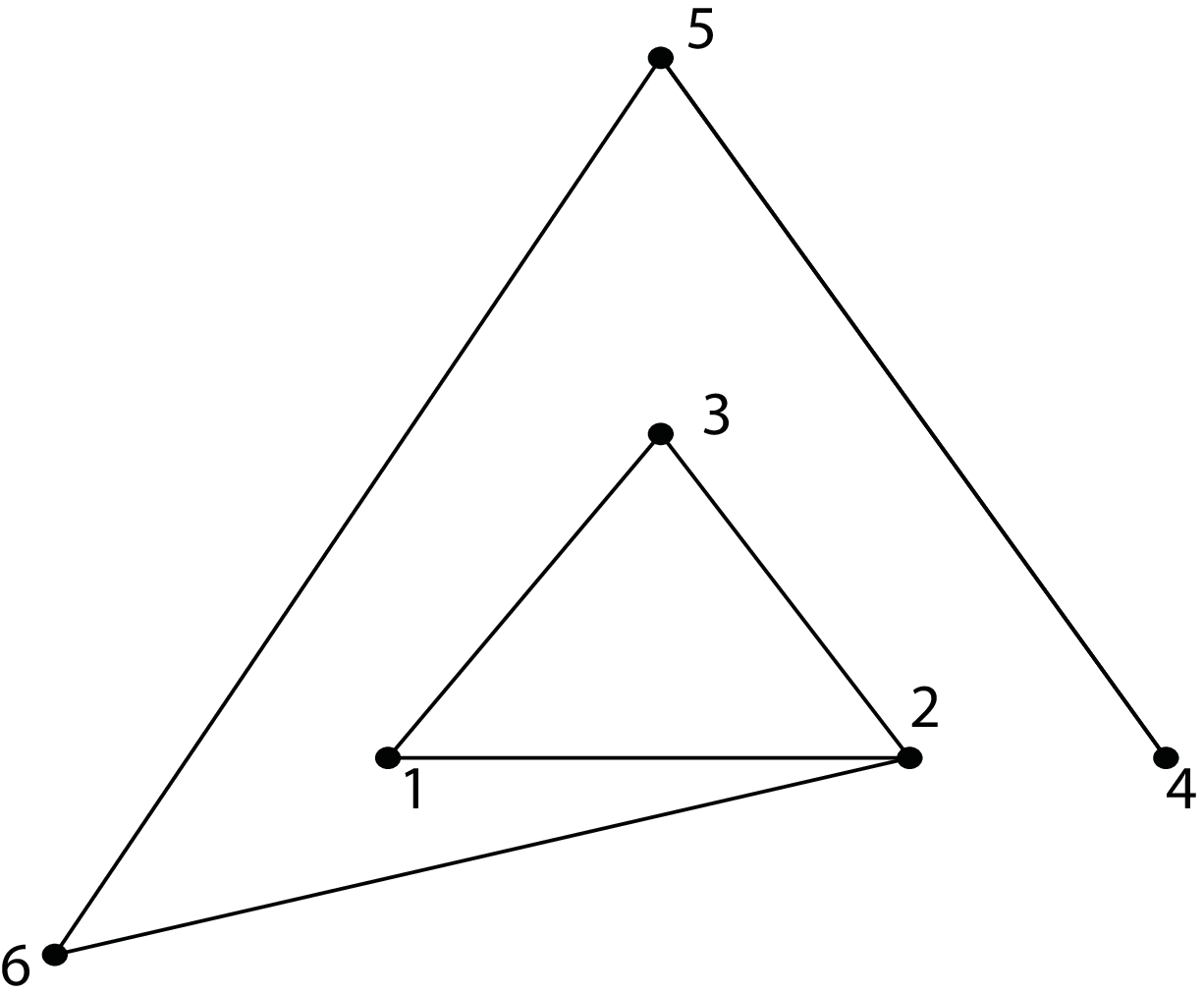}}
  \end{center}
  \caption{The graph $L$ on the right  is obtained as the difference of $K$ and $H$, where $H=\{\{2,4\}, \{3,5\},\{1,6\}\}$ }
\end{figure} 

An interesting consequence of the previous corollary is the relation
between  the smallest non-zero eigenvalues of the graphs $L$ and $K$ for certain types of graphs $H$.
\begin{coll}
Let $(K,w_{K})$ be a weighted graph, $(H,w_{H})$ its subgraph,    and $(L,w_{L})$ their proper difference.
Assume the graph $H=(v_{i},v_{j})$ is an edge with one pending vertex
i.e., the only neighbour of the vertex $v_{j}$ is $v_{i}$, and let 
$w_{H}(v_{i},v_{j})=w_{K}(v_{i},v_{j})$. Then
$\lambda_{k-1}\leq \theta_{k}$ for every $k\leq N$. In particular, 
$\lambda_{2}\leq \theta_{3}$, where $\lambda_{2}$ and $\theta_{3}$ are the smallest non-zero eigenvalues of $\mathcal{L}^{up}_{0}(K)$ and $\mathcal{L}^{up}_{0}(L)$,  respectively.
\end{coll}
The above result implies  that the algebraic connectivity of a graph increases with the removal of pending vertices.
A similar claim holds for the general Laplace operator $\mathcal{L}^{up}_{n}$.
\begin{coll}
\label{Corollary collapses}
Let $(K,w_{K})$ be a weighted simplicial complex, and $(H,w_{H})$ its subcomplex such that   $(L,w_{L})$ is their proper difference.
Assume the graph $H$ is an $(n+1)$-simplex $\bar{F}$ with one pending vertex $v$, i.e., the only $(n+1)$-face containing $v$ is $H$ itself, and 
$w_{H}(\bar{F})=w_{K}(\bar{F})$. Then, 
$\lambda_{k-1}\leq \theta_{k}$ for every $k\leq N$. In particular,
$\lambda_{K-1}\leq \theta_{K}$, where   $\lambda_{K-1}$ and $\theta_{K}$ are the smallest non-zero eigenvalues of $\mathcal{L}^{up}_{n}(K)$ and $\mathcal{L}^{up}_{n}(L)$,  respectively.
\end{coll}
\begin{proof}
From \cite{HorakNCL} (Theorem 3.1.), it follows that the number of zero eigenvalues in the spectrum of $\mathcal{L}_{n}^{up}(K)$
 is $\dim C^{n}(K,\rb)+\dim H^{n+1}(K,\rb)-\dim C^{n+1}(K,\rb)$.
The number of zeros in the spectrum of the generalized upper Laplacian
of the simplicial complex $L'$, which is obtained from the simplicial complex $L$ by deletion of simplices of weight $0$ is
$$\dim C^{n}(K,\rb)-(n+1)+\dim H^{n+1}(K,\rb)-\dim C^{n+1}(K,\rb)+1=\dim C^{n}(K,\rb)+\dim H^{n+1}(K,\rb)-\dim C^{n+1}(K,\rb)-n.$$
Hence, the number of zeros in $\mathcal{L}^{up}_{n}(L)$ is $$\dim C^{n}(K,\rb)+\dim H^{n+1}(K,\rb)-\dim C^{n+1}(K,\rb)-n+(n+1),$$ since 
$L$ contains exactly $n+1$ $n$-faces whose weight is $0$.
Therefore, there is an additional  zero in the spectrum of $\mathcal{L}^{up}_{n}(L)$ compared to  $\mathcal{L}^{up}_{n}(K)$. Together with  Theorem \ref{main theorem} this proves the claim.
\end{proof}

In the sequel, we  shall further exploit the properties of $\mathcal{R}_{\pi^{*}\pi}$ and $\mathcal{R}_{i^{*}i}$, and  obtain the following  inequalities.
\begin{thm}
\label{Proposition interlacing}
\begin{equation}
\label{left-right inequality interlacing}
\min_{\substack{\bar{F}\in S_{n+1}(L):\\ w_{L}(\bar{F})\neq 0}}\frac{w_{L}(\bar{F})}{w_{K}(\bar{F})} \lambda_{k-D_{\mathcal{Z}}}\leq \theta_{k}\leq \lambda_{k} \max_{\substack{F\in S_{n}(L):\\ w_{L}(F)\neq 0}}\frac{w_{K}(F)}{w_{L}(F)},
\end{equation}
 where   $D_{\mathcal{Z}}=\dim C^{n}(L,\rb)-\dim C^{n+1}(L',\rb)+\dim H^{n+1}(L',\rb)$, and 
 $L'$ is a maximal subcomplex of $L$ whose simplices are of non-zero weight, and 
 $\lambda_{0}=\ldots=\lambda_{-D_{\mathcal{Z}}+1}=0$.
\end{thm}
\begin{proof} 
From (\ref{eq1}), it follows that $\mathcal{R}_{\pi^{*}\pi}(\delta_{K}ig)=0$ iff $\delta_{K}ig(\bar{F})=0$ for every $\bar{F}\in S_{n+1}( L')$, i.e.,
the restriction of $\delta_{K}ig(\bar{F})$ on $C^{n+1}(L',\rb)$ must be zero.
Let $\bar{L}$ be a subcomplex of $K$, such that $S_{n}(K)=S_{n}(\bar{L})$ and $S_{n+1}(L')=S_{n+1}(\bar{L})$, and let $\mathcal{Z}=\ker \delta_{\bar{L}}$, where $\delta_{\bar{L}}$ is the $n$-th coboundary map of the simplicial complex $\bar{L}$.
Then, 
\begin{equation}
\label{bar L implictaion}
\textrm{if }  g \in \ker \delta_{\bar{L}}, \textrm{then }\mathcal{R}_{\pi^{*}\pi}(\delta_{K}ig)=0.
\end{equation}
The dimension of $\mathcal{Z}$ is $\dim C^{n}(L,\rb)-\dim C^{n+1}(L',\rb)+\dim H^{n+1}(L',\rb)$.
From   (\ref{ineq01}) and (\ref{bar L implictaion}), we obtain
\begin{align}
\theta_{k}&\geq \min_{\mathcal{V}_{k}}\max_{g\in \mathcal{V}_{k}}  \mathcal{R}_{\pi^{*}\pi}(\delta_{K}ig)   \mathcal{R}_{\mathcal{L}^{up}_{n}(K)}(ig)\notag\\
&\geq \min_{\mathcal{V}_{k}}\max_{g\in \mathcal{V}_{k}, g\perp \mathcal{Z}}  \mathcal{R}_{\pi^{*}\pi}(\delta_{K}ig)   \mathcal{R}_{\mathcal{L}^{up}_{n}(K)}(ig)\label{ineq smaller set}\\
&\geq \min_{\substack{\bar{F}\in S_{n+1}(L):\\ w_{L}(\bar{F})\neq 0}}\frac{w_{L}(\bar{F})}{w_{K}(\bar{F})} \min_{\mathcal{V}_{k}}\max_{g\in \mathcal{V}_{k}, g\perp \mathcal{Z}}     \mathcal{R}_{\mathcal{L}^{up}_{n}(K)}(ig)\label{ineq min }\\
&\geq \min_{\substack{\bar{F}\in S_{n+1}(L):\\ w_{L}(\bar{F})\neq 0}}\frac{w_{L}(\bar{F})}{w_{K}(\bar{F})} \min_{\mathcal{V}_{k}}\max_{g\in \mathcal{V}_{k-D_{\mathcal{Z}}}}     \mathcal{R}_{\mathcal{L}^{up}_{n}(K)}(ig)\label{ineq final }\\
&=\min_{\substack{\bar{F}\in S_{n+1}(L):\\ w_{L}(\bar{F})\neq 0}}\frac{w_{L}(\bar{F})}{w_{K}(\bar{F})}\lambda_{k-D_{\mathcal{Z}}}\notag .
\end{align}
The derivation of the inequalities above  follows  as in Theorem \ref{main theorem}.\\
 The right inequality in (\ref{left-right inequality interlacing}) is due to  (\ref{ineqZX}) and the fact that $\mathcal{R}_{i^{*}i}(f)\leq \max_{F\in S_{n}(L)}\frac{w_{K}(F)}{w_{L}(F)}$.
\end{proof}
A special case of the previous theorem  is  the Courant-Weyl inequality for the combinatorial Laplacian.
\begin{coll}
\label{Courant-Weyl}
Let $(K,w_{K})$ be a weighted simplicial complex and $(L,w_{L})$ its subcomplex, such that 
$C^{n}(K,\rb)=C^{n}(L,\rb)$, $w_{K}=1$, and $w_{L}(F)\in\{0,1\}$. Then 
$\theta_{k}\leq\lambda_{k}$, where 
$\theta_{i}$ and $\lambda_{i}$ are the eigenvalues of $L^{up}_{n}(L)$ and $L^{up}_{n}(K)$, respectively, ordered increasingly.
\end{coll}
Therefore, inequality (\ref{Godsil interlacing}) is a direct consequence of Corollary \ref{Courant-Weyl} and Theorem \ref{main theorem}.
In the sequel, we shall derive an upper bound for the maximal eigenvalue of the generalized Laplacian, using the above  interlacing theorems.
\begin{coll}
Let $(L,w_{L})$ be a weighted simplicial complex on $N$ vertices, and let $\theta_{N_{L}}$ be the maximum eigenvalue of $\mathcal{L}^{up}_{n}(L)$. 
Then 
\begin{equation}
\label{generalized inequality}
\theta_{N_{L}}\leq \frac{N}{\min_{F\in S_{n}(L)} w_{L} (F)}.
\end{equation}
In particular, if $w_{L}\equiv 1$ (the  combinatorial  Laplacian), then 
\begin{equation}
\label{Duval inequality}
\theta_{N_{L}}\leq N.
\end{equation}
\end{coll}
\begin{proof}
Without  loss of generality assume that the weight function $w_{L}$ on
every simplex has value $\le 1$. Let $K_{N}$ be an $(N-1)$-simplex on 
$N$ vertices with the weight function $w_{K}\equiv 1$. Then  it is
possible to obtain any $(L,w_{L})$ as a difference
$(K_{N},w_{K})-(H,w_{H})$ for some subcomplex $(H,w_{H})$ of $(K_{N},w_{K})$.
The maximum eigenvalue of $\mathcal{L}^{up}_{n}(K_{N})$ is $\lambda_{N}=N$. By combining this result with (\ref{left-right inequality interlacing}),  we obtain the  inequalities  (\ref{generalized inequality}) and (\ref{Duval inequality}).
\end{proof}
\begin{rem}
In \cite{DuvalShifted} Duval and Reiner obtained (\ref{Duval
  inequality}) by using a different method. For the
normalized Laplacian, the operator estimate   (\ref{generalized inequality}) is better than $\theta_{N_{L}}\leq (n+2)$   if $\frac{N}{(n+2)}\leq \min_{F\in S_{n}(L)} \deg (F)$.
\end{rem}

\section{Interlacing effects of covering and simplicial maps}
\label{section 3}
In this section we analyze the effect of covering maps and more general simplicial maps on the spectrum of the Laplace operator.
Furthermore, we propose a  definition for discrete coverings, different from that in \cite{Rotman}, \cite{Gustavson}, which agrees with the continuous counterpart of covering maps. In particular, we give counterexamples to the  Universal Lifting Theorem for discrete covering maps from \cite{Rotman} (Theorem 2.1.), and Theorem 4.4.
from \cite{Gustavson}, and propose a modified definition of coverings, called strong covering, which fixes the problem.

\begin{defn}
Let $K$ and $L$ be simplicial complexes.
A \emph{simplicial map} $\varphi: K\rightarrow L$  is a function  from the vertices of $K$ to the vertices  of $L$, such that $\varphi(v_{0}),\varphi(v_{1}),\ldots, \varphi(v_{n})$ span a simplex in $L$ whenever $v_{0},v_{1},\ldots,v_{n}$ span a simplex in $K$.
\end{defn}
The  homomorphism of chain groups  $\varphi_{n}: C_{n}( K,\rb)\rightarrow C_{n}(L,\rb)$,  induced by  a simplicial map  $\varphi: K\rightarrow L$, is defined $\rb$-linearly by extending the following map on basis elements 
  $\varphi_{n}[v_{0},v_{1},\ldots,v_{n}]=[\varphi(v_{0}),\varphi(v_{1}),\ldots, \varphi(v_{n})]$.
If $\varphi(v_{i})=\varphi(v_{j})$, for some $i\neq j$, then  $\varphi_{n}[v_{0},v_{1},\ldots,v_{n}]=0$.
By the duality principle the induced homomorphism  $\varphi^{n}: C^{n}( L,\rb)\rightarrow C^{n}(K,\rb)$  of the cochain groups is
 $\varphi g
 ([v_{0},v_{1},\ldots,v_{n}])=g([\varphi(v_{0}),\varphi(v_{1}),\ldots,
 \varphi(v_{n})])$. For brevity, in what follows we shall omit  the index $n$ in 
 $\varphi_{n}$ and $\varphi^{n}$, and assume that all faces, which are not written as a set of their vertices, are oriented positively, unless stated otherwise. 

Let $[G]$ and   $[F]$ be  positively oriented  $n$-simplices
on the vertex sets  $v_{0},v_{1},\ldots,v_{n}$ and $\varphi(v_{0}),\varphi(v_{1}),\ldots, \varphi(v_{n})$, respectively.
We can also represent the induced chain map  $\varphi$ as $\varphi [F]=\sgn( [F],[G]) [G]$, where 
$$
\sgn( [F],[G])=\left\{\begin{array}{rl}
1 & \textrm{ if } [\varphi(v_{0}),\varphi(v_{1}),\ldots,\varphi(v_{n})] \textrm{  is a positively oriented simplex},\\
-1 & \textrm{ otherwise.}
\end{array}
\right.
$$
Henceforth,   $\varphi e_{[G]}$ can be written as 
 $\varphi e_{[G]}=\sum_{\substack{F\in S_{n}(K):\\ \varphi(F)=G}}\sgn( [F],[G]) e_{[F]}$.
 Assume $w_{K}$ and $w_{L}$ are weight functions assigned  to $K$ and $L$. Then, the adjoint 
$\varphi^{*}$ of a simplicial map $\varphi$ is given by  $\varphi^{*}e_{[F]}=\sgn( [F],[G]) \frac{w_{K}(F)}{w_{L}(G)}e_{[G]}$.
However, we will not be concerned with the Laplace operators $\mathcal{L}^{up}_{n}(L,w_{L})$ for an arbitrary weight function $w_{L}$, 
instead we will consider the  weight function $w_{L}$, induced by the
simplicial map $\varphi$ and the weight function $w_{K}$.
There are two reasonable  ways to define  $w_{L}$, given  $w_{K}$ and
the simplicial map $\varphi$. Namely, 
\begin{equation}
\label{rule (ii)}
w_{L}(G)=\sum_{\substack{F\in S_{n}(K):\\ \varphi(F)=G}} w_{K}(F) \textrm{, for every } G\in S_{n}(L),
\end{equation}
and\footnote{Note that it  suffices to consider the weights of the $n$
  and $n+1$-faces of the simplicial complexes $K,L$ in order to completely determine the eigenvalues of $\mathcal{L}^{up}_{n}$.},
 \begin{equation}
\label{rule (i)}
 w_{L}(G)=\sum_{\substack{F\in S_{n}(K):\\ \varphi(F)=G}}w_{K}(F)-\sum_{\substack{\bar{F}\in S_{n+1}(K):\varphi(\bar{F})=G}}w_{K}(\bar{F}),
 \textrm{ for } G\in S_{n}(L).
\end{equation}
In the remainder we  prove that  the following diagram
\begin{equation}
\label{diagram simplicial map}
\begin{CD}
C^{n}(K,\rb)  @<\mathcal{L}^{up}(K)<<   C^{n}(K,\rb) \\
 @VV\varphi^{*} V    @AA\varphi A \\
C^{n}(L,\rb)  @<\mathcal{L}^{up}(L)<<   C^{n}(L,\rb) \\
\end{CD}
\end{equation}
 is  commutative for a  simplicial map $\varphi$, where $w_{L}$ is determined in accordance with the  rule (\ref{rule (ii)}).
In particular,  we will prove that 
\small  
\begin{equation}
\label{varphi equation}
\varphi^{*} \mathcal{L}^{up}(K)\varphi e_{[G]}=\sum_{\substack{F\in S_{n}(K):\\ \varphi(F)=G}}\; \sum_{\substack{F'\in S_{n}(K):\\ F,F'\in  \partial \bar{F}}} \sgn( [F],[G])\sgn([F],\partial [\bar{F}])\sgn( [F'],[G'])\sgn([F]',\partial [\bar{F}])\frac{w_{K}(\bar{F})}{w_{L}(G')}e_{[G']},
\end{equation}
\normalsize
where $G'$ denotes $\varphi(F')$, and is equal to $\mathcal{L}^{up}_{n}(L)$.
For simplicity, we first  look at the case when $\varphi$ is a covering map.
According to Rotman (see \cite{Rotman}), a covering complex is defined as follows.
\begin{defn}
\label{definition covering}
Let $K$ be a simplicial complex. A pair$(K,\varphi)$ is a covering complex of $L$  if:
\begin{itemize}
\item[(i)] $K$ is a connected simplicial complex, 
\item[(ii)] $\varphi: K\rightarrow L$ is a simplicial map, and
\item[(iii)] for every simplex $G\in L$, $\varphi^{-1}(G)$ is a union of pairwise disjoint simplices $\varphi^{-1}(G)=\bigcup_{i} F_{i}$, with 
a bijection $\varphi\lvert_{F_{i}}:F_{i}\rightarrow G$  for each $i$.
A simplicial map $\varphi$ is called a \emph{covering map}.
\end{itemize}
\end{defn}
\label{definition topological covering }
The covering complexes are meant to discretize the notion of covering topological spaces.
\begin{defn}
Let $Y$ be a topological space. A covering space of $Y$ is a topological space $X$ together with a continuous surjective map
$p:X\rightarrow Y$, 
such that 
\begin{itemize}
\item[(i)] for every $y\in Y$, there exists an open neighbourhood $U$, such that preimage  $p^{-1}(U)$ of $U$, is a  union of pairwise disjoint open sets $p^{-1}(U)=\bigcup_{i} V_{i}$, such that $p\lvert_{V_{i}}:V_{i}\rightarrow U$ is a homeomorphism for every $i$. 
\end{itemize}
\end{defn}

\begin{exmp}
\label{example coverings}
All horizontal  pairs in  Figure \ref{coverings}  represent coverings. In particular, the union of the  simplicial complexes 
$\tilde{K}=\{\{1,6\}, \{2,6\}, \{3,5\}, \{3,4\}, \{1\}, \{2\},\{3\},\{4\},\{5\},\{6\} \}$ in Figure \ref{cov01}
is a covering space of their join  $K=\{\{1,6\}, \{2,6\}, \{3,5\}, \{3,4\}, \{1\}, \{2\},\{3\},\{4\},\{5\},\{6\} \}$
shown  in  Figure \ref{cov02}.
The hexagon  $\{1,2\}, \{2,3\}, \{3,4\}, \{4,5\}, \{5,6\}, \{6,1\}$ is a covering space of 
the hollow triangle  $\{1',2',3'\}$, and  the covering map $\varphi$, is given by 
$\varphi(\{1\})=\varphi(\{4\})=\{1'\}$, $\varphi(\{2\})=\varphi(\{5\})=\{2'\}$, and $\varphi(\{3\})=\varphi(\{6\})=\{3'\}$.
The simplicial complex $\tilde{L}=\{\{1,2\}, \{2,3\}, \{3,4\}, \{4,5\}, \{1\}, \{2\},\{3\},\{4\},\{5\} \}$ in Figure \ref{cov05} is a covering of 
 $L=\{\{1',2'\}, \{2',3'\}, \{1',3'\}, \{1',5'\}, \{1'\}, \{2'\},\{3'\},\{5'\}\}$. The covering map is given by
 $\varphi(\{1\})=\varphi(\{4\})=\{1'\}$,  $\varphi(\{2\})=\{2'\}$,  $\varphi(\{3\})=\{3'\}$ and  $\varphi(\{5\})=\{5'\}$.
\end{exmp}
Note that in the case of covering maps the weight functions in (\ref{rule (ii)}) and  (\ref{rule (i)}) are identical.
  \begin{figure}[h!tp]
      \label{coverings}
  \begin{center}
 \subfigure[raggedright,scriptsize][ ]{\label{cov01}\includegraphics[scale=0.4]{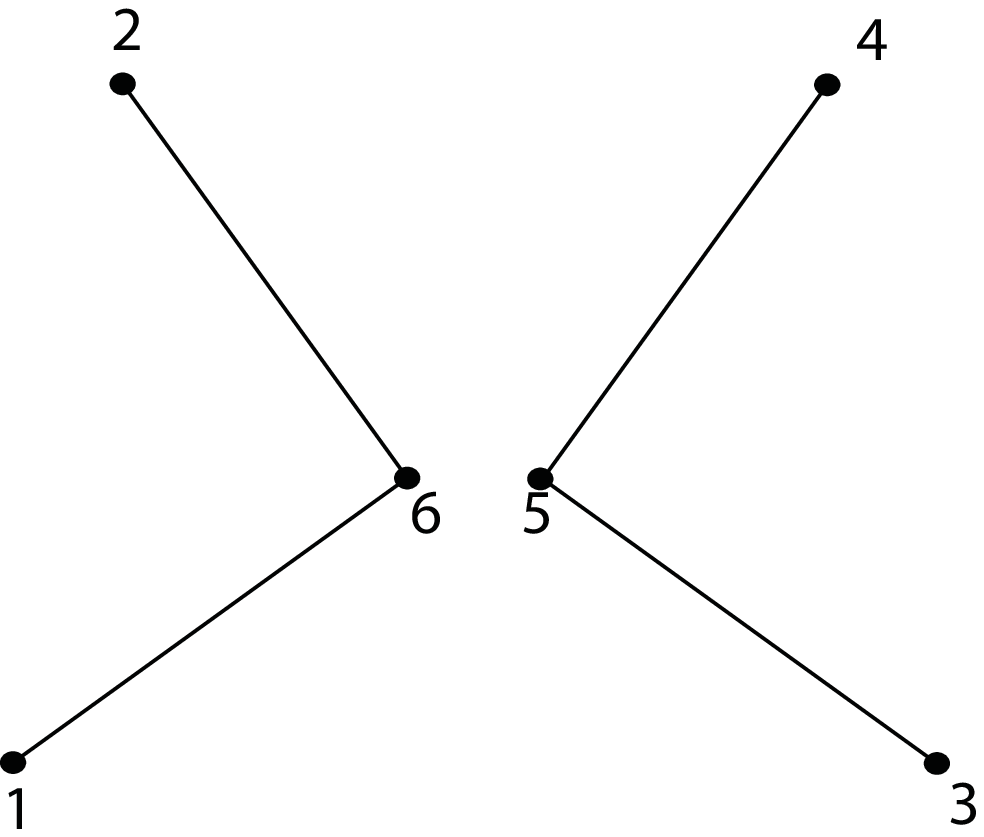}} \qquad\qquad
    \subfigure[raggedright,scriptsize][  ]{\label{cov02}\includegraphics[scale=0.4]{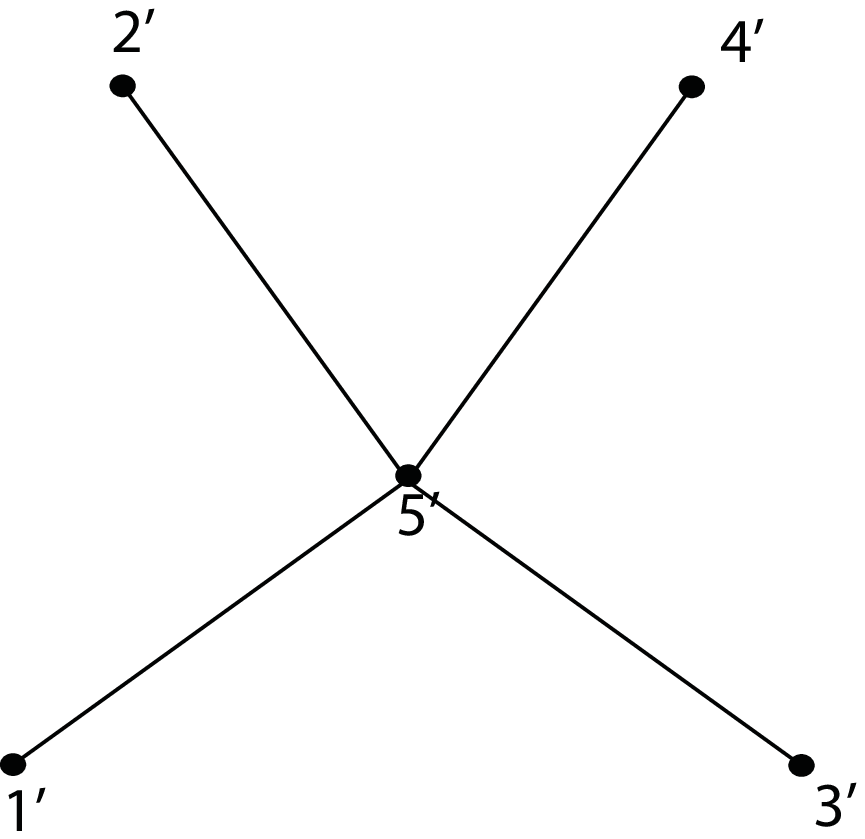}}\\
    \subfigure[raggedright][   ]{\label{cov03}\includegraphics[scale=0.4]{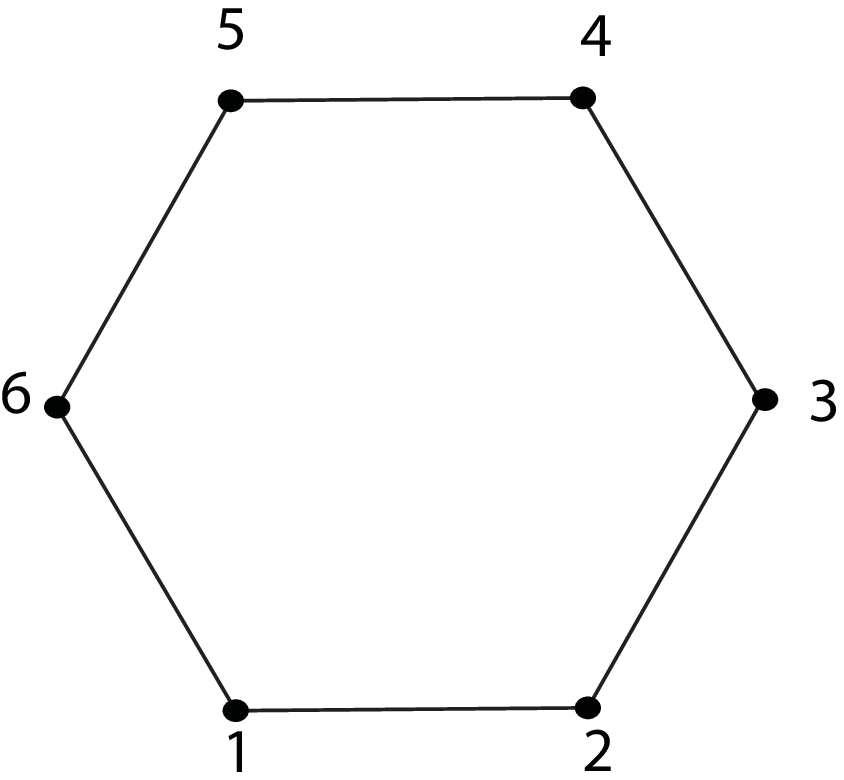}}\qquad\qquad
    \subfigure[raggedright][ ]{\label{cox04}\includegraphics[scale=0.45]{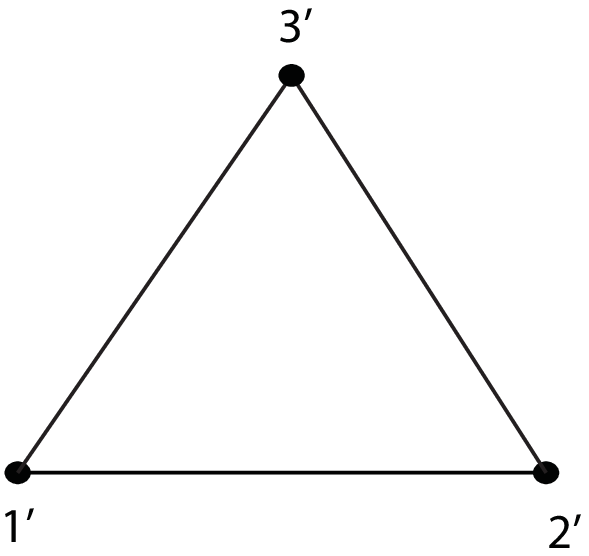}}\\
    \subfigure[raggedright][   ]{\label{cov05}\includegraphics[scale=0.4]{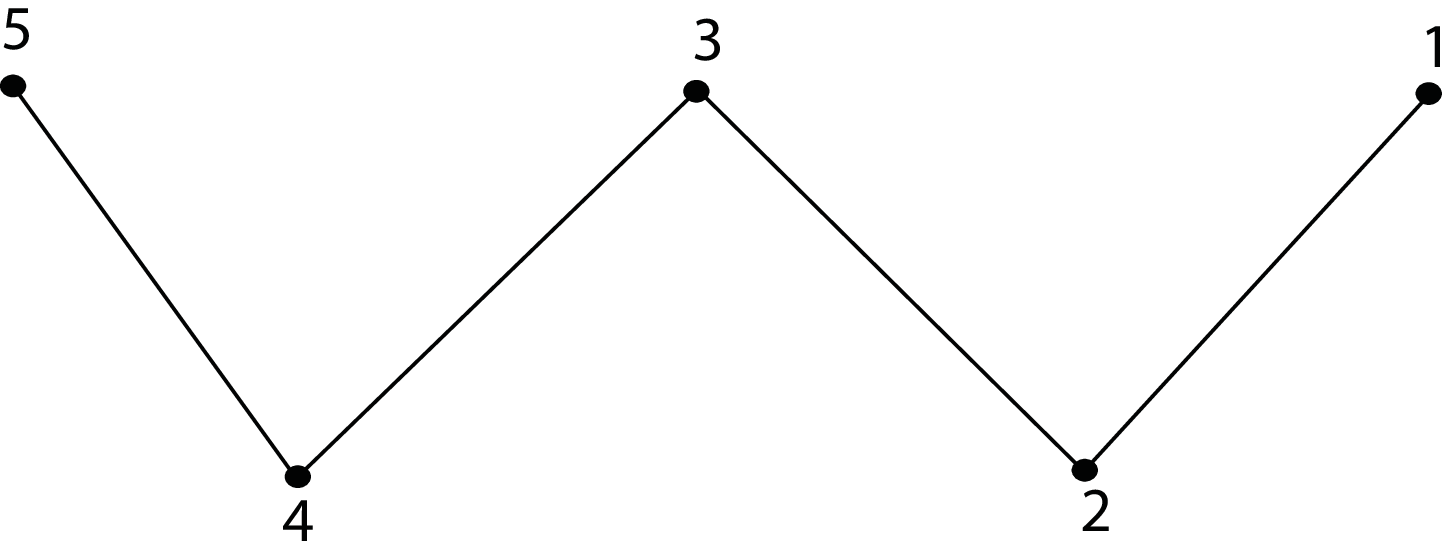}}\qquad
    \subfigure[raggedright][   ]{\label{cov06}\includegraphics[scale=0.4]{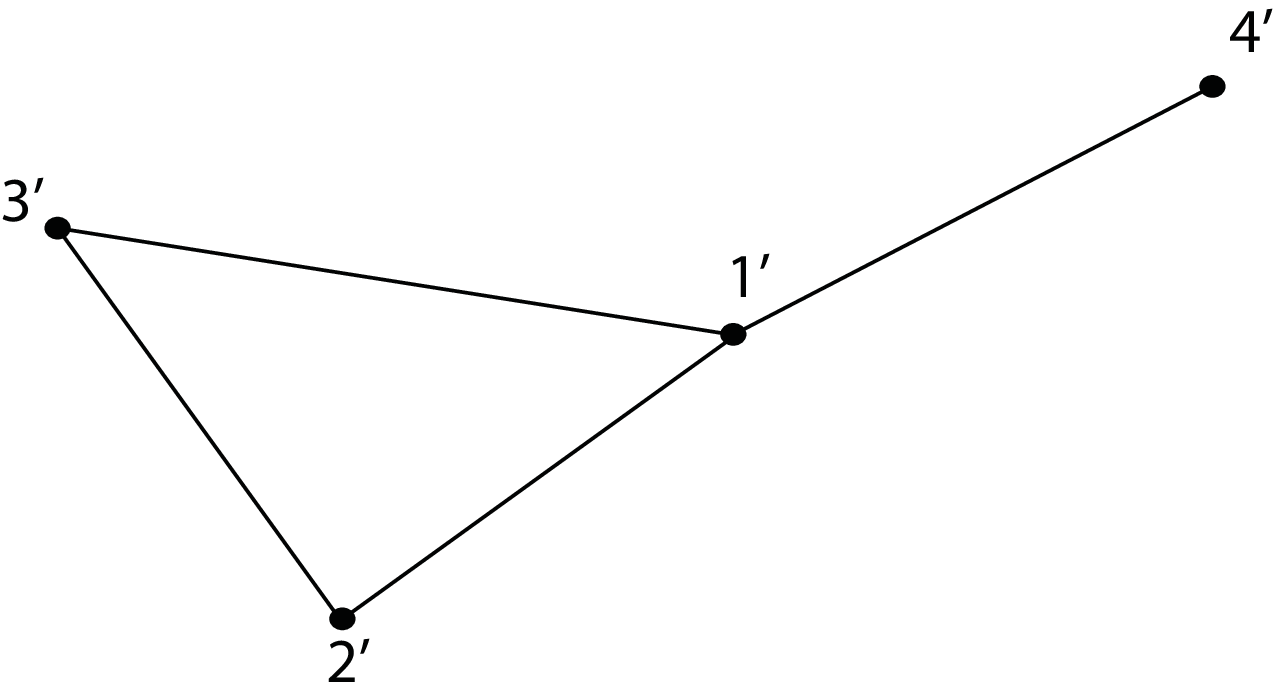}}
  \end{center}
  \caption{Coverings}
\end{figure} 

\begin{lem}
Let $\varphi:K\rightarrow L$ be a covering map, and let $w_{K}$ and $w_{L}$ be weight functions  satisfying (\ref{rule (ii)}).
Then, diagram (\ref{diagram simplicial map}) commutes.
\end{lem}
\begin{proof}
Let $[\bar{F}]=[v_{0},\ldots,v_{n+1}]$ be a positively  oriented $(n+1)$-simplex in $K$, and   $[F]=[v_{0},\ldots,\hat{v_{i}},\ldots v_{n+1}]$,  $[F']=[v_{0},\ldots,\hat{v_{j}},\ldots v_{n+1}]$ its oriented facets. 
Let $\bar{G}=\{\varphi(v_{0}),\ldots, \varphi(v_{n+1})\}$, $G=\{\varphi(v_{0}),\ldots, \hat{\varphi(v_{i})},\ldots,\varphi(v_{n+1})\}$, and 
$G'=\{\varphi(v_{0}),\ldots, \hat{\varphi(v_{j})},\ldots,\varphi(v_{n+1})\}$ be the images of $\bar{F},F$, and $,F'$ under the covering map $\varphi$,  respectively. Assume $[\bar{G}]$, $[G]$, and $[G']$ are positively oriented. Then, 
$$[\bar{G}]=[\varphi(v_{0}),\ldots, \varphi(v_{n+1})]\sgn([\bar{F}],[\bar{G}]),$$  $$[G]=[\varphi(v_{0}),\ldots, \hat{\varphi(v_{i})},\ldots,\varphi(v_{n+1})]\sgn([F],[G])$$ and $$[G']=[\varphi(v_{0}),\ldots, \hat{\varphi(v_{j})},\ldots,\varphi(v_{n+1})]\sgn([F],[G]).$$
Therefore,
\small
\begin{align}
\label{equation signs}
\sgn([G],\partial [\bar{G}])=& \sgn\left ([\varphi(v_{0}),\ldots, \hat{\varphi(v_{i})},\ldots,\varphi(v_{n+1})]\sgn([F],[G]), [\varphi(v_{0}),\ldots, \varphi(v_{n+1})]\sgn([\bar{F}],[\bar{G}])\right)\\
=& (-1)^{i}\sgn([F],[G])\sgn([\bar{F}],[\bar{G}])\\
=& \sgn([F],\partial [\bar{F}])\sgn([F],[G])\sgn([\bar{F}],[\bar{G}]).
\end{align}
\normalsize
Hence, (\ref{varphi equation}) equals to
\begin{align}
\varphi^{*} \mathcal{L}^{up}(K)\varphi e_{[G]}=&\sum_{\substack{F\in S_{n}(K):\\ \varphi(F)=G}}\;\sum_{\substack{F'\in S_{n}(K):\\ F,F'\in  \partial \bar{F}}} \sgn([G],\partial [\bar{G}])\sgn([G'],\partial [\bar{G}])\frac{w_{K}(\bar{F})}{w_{L}(G')}e_{[G']}\label{eqX}\\
=&\sum_{\substack{F\in S_{n}(K):\\ \varphi(F)=G}}\;
\sum_{\substack{\bar{F}\in S_{n+1}(K):\\ F\in \partial \bar{F}}} w_{K}(\bar{F}) \sum_{F': F'\in \partial \bar{F} } \sgn([G],\partial [\bar{G}])\sgn([G'],\partial [\bar{G}])\frac{1}{w_{L}(G')}e_{[G']}\label{contraction equation},
\end{align}
where $\varphi(F')=G'$.
Since $w_{L}(G)=\sum_{\substack{F\in S_{n}(K):\\ \varphi(F)=G}}w_{K}(F)$,  from (\ref{eqX}) we get
\begin{align}
\varphi^{*} \mathcal{L}^{up}(K)\varphi e_{[G]}
=&\sum_{\substack{G'\in S_{n}(L):\\ G,G'\in \partial \bar{G}}}
\sgn([G],\partial [\bar{G}])\sgn([G]',\partial [\bar{G}])  \frac{w_{L}(\bar{G})}{w_{L}(G')}e_{[G']}\\
=&\mathcal{L}^{up}e_{[G]}.
\end{align}
Thus,  diagram (\ref{diagram simplicial map}) commutes.
\end{proof}

Let $\mathcal{V}_{k}$ be a $k$-dimensional subspace of $C^{n}(L,\rb)\cong\rb^{N_{L}}$,  then according to Theorem \ref{min-max theorem} we have
\begin{align}
\theta_{k}=&\max_{\mathcal{V}_{k-1}}\min_{g\perp \mathcal{V}_{k-1}}\frac{(\mathcal{L}^{up}(L)g,g)}{(g,g)}\\
=&\max_{\mathcal{V}_{k-1}}\min_{g\perp \mathcal{V}_{k-1}}\frac{(\varphi^{*}\mathcal{L}^{up}(K)\varphi g,g)}{(g,g)}\\
=&\max_{\mathcal{V}_{k-1}}\min_{g\perp \mathcal{V}_{k-1}}\frac{(\mathcal{L}^{up}(K)\varphi g,\varphi g)}{(\varphi g,\varphi g)}\label{ineq M01}.
\end{align}
Equality (\ref{ineq M01}) is due to 
\begin{align*}
(\varphi g,\varphi g) =& \sum_{G\in S_{n}(L)}\sum_{\substack{F\in S_{n}(K):\\ \varphi(F)=G}}(\sgn([F],[G])g(\varphi F))^{2}w_{K}(F)\\
=& \sum_{G\in S_{n}(L)} g(G)^{2}w_{L}(G)=(g,g).
\end{align*}
Let $\mathcal{W}$ be a vector space generated by 
$\{\sgn([F_{ji}],[G_{i}])e_{[F_{ji}]}-\sgn([F_{(i+1)j}],[G_{i}])e_{[F_{(i+1)j}]}\mid \bigcup_{j} F_{ji}=\varphi^{-1}(G_{i}), \textrm{ and  }   \bigcup_{i} G_{i}=S_{n}( L)\}$. 
The dimension of  $\mathcal{W}$ is $N_{K}-N_{L}$, where $N_{K}$,
$N_{L}$  is the number of $n$-dimensional simplices in the simplicial complexes $K$ and  $L$, respectively.
Let $\varphi g=f$ and let $\mathcal{V}_{k-1}$ be  a subspace of $\rb^{N_{K}}$.  From (\ref{ineq M01}) we get
\begin{align}
\theta_{k}=&\max_{\mathcal{V}_{k-1}}\min_{f\perp \mathcal{V}_{k-1}, f\perp \mathcal{W}}\frac{(\mathcal{L}^{up}(K)f,f)}{(f,f)}\label{ineq M02}\\
\geq & \max_{\mathcal{V}_{k-1}}\min_{f\perp \mathcal{V}_{k-1} }\frac{(\mathcal{L}^{up}(K)f,f)}{(f,f)}\label{ineq M03}\\
\geq & \lambda_{k}.
\end{align}
Inequality (\ref{ineq M03}) holds since we are minimizing over a larger set.
As a broader optimization,  from (\ref{ineq M02}) we obtain the upper interlacing inequality, i.e.,
\begin{align}
\theta_{k}\leq & \max_{\mathcal{V}_{k-1+N_{K}-N_{L}}}\min_{f\perp \mathcal{V}_{k-1+N_{K}-N_{L}}}\frac{(\mathcal{L}^{up}(K)f,f)}{(f,f)}\label{specCase}\\
=& \lambda_{k+N_{K}-N_{L}},
\end{align}
and we obtain the following theorem.
\begin{thm}
\label{Interlacing covering}
Let $\varphi: K\rightarrow L$ be a covering map, and let $w_{K}$ and $w_{L}$ be weight functions of $K$ and $L$, respectively, such that 
(\ref{rule (ii)}) holds.
Let $\lambda_{1}\leq \lambda_{2}\leq \ldots \leq \lambda_{N_{K}}$ and
$\theta_{1}\leq \theta_{2}\leq \ldots \leq \theta_{N_{L}}$ be the eigenvalues
of $\mathcal{L}^{up}_{n}(K)$ and $\mathcal{L}^{up}_{n}(L)$, respectively.
Then,
\begin{equation}
\lambda_{k}\leq\theta_{k}\leq\lambda_{k+N_{K}-N_{L}},
\end{equation}
where $\lambda_{N_{K}+1}=\ldots=\lambda_{2N_{K}-N_{L}}=N$, where $N$ is the number of vertices of  $K$.
\end{thm}
Note that, however, covering complexes are an inaccurate
discretization of covering spaces (continuous setting), since
Definition \ref{definition covering} does not contain a discrete analogue of the  homeomorphic neighbourhood requirement (i) from Definition \ref{definition topological covering }. As an example see Figure \ref{cov05}  and Figure \ref{cov06}.
In what follows, we provide the definition of a strong covering, which
accurately discretizes the notion of covering from the continuous setting and overcomes this problem. 
\begin{defn}
\label{definition strong covering}
A covering map $\varphi:K\rightarrow L$ is called a \emph{strong} covering if 
for every $G$, which is a facet of $\bar{G}$ and every $F\in \varphi^{-1}(G)$, there exists $\bar{F}\in S_{n+1}(K)$ such that 
$F\in \partial \bar{F}$ and $\varphi(\bar{F})=\bar{G}$.
\end{defn}
\begin{rem}
\label{remark strong covering}
According to the previous definition for any two $n$-faces of $L$,   $G$ and $G'$, which  are $(n+1)$-up neighbours, and for every 
$F\in \varphi^{-1}(G)$,  there must exist  $F'\in S_{n}(K)$ such that 
$F$ and $F'$ are  $(n+1)$-up neighbours and $\varphi(F')=G'$.
\end{rem}
\begin{lem}
\label{lemma coverings}
Let $\varphi: K\rightarrow L$ be a strong covering. Then, for every  $F\in S_{n}(K)$ and every 
$G\in S_{n}(L)$, such that $\varphi(F)=G$, the following holds
\[
\lvert\{ \bar{F}\in S_{n+1}(K)\mid F\in \partial \bar{F}\} \lvert= \lvert\{ \bar{G}\in S_{n+1}(L)\mid G\in \partial \bar{G}\} \lvert.
\]
\end{lem}
\begin{proof}
Assume $\bar{F}_{1},\ldots, \bar{F}_{k}\in S_{n+1}(K)$, such that $ F\in \partial \bar{F}_{i}$, $i\in \{1,\ldots,k\}$. Then, due to the definition of a covering (Definition \ref{definition covering}), 
$\varphi(\bar{F}_{i})\in S_{n+1}(L)$ and $ \varphi(F)\in \partial \varphi(\bar{F}_{i})$, and  $\varphi(\bar{F}_{i})\neq \varphi(\bar{F}_{j}) $, for every 
$i\neq j$.
Assume $\bar{G}\in S_{n+1}(L)$, such that $ G\in \partial
\bar{G}$. Since $\varphi$ is a strong covering and satisfies
Definition \ref{definition strong covering}, then there exists
$\bar{F}\in S_{n+1}(K)$, such that $ F\in \partial \bar{F}$, such that
$\varphi(\bar{F})=\bar{G}$. In addition, 
if $ G\in \partial \bar{G},\partial \bar{G'} $ and $\bar{G}\neq \bar{G'}$,
 then  $\bar{F}\neq \bar{F'}$ (simply because $\varphi$ is a map from $K$ to $L$).
\end{proof}
\begin{lem}
\label{lemma constant coverings}
Let $\varphi: K\rightarrow L$ be a strong covering, then there exists a constant $c\in \mathbb{N}$, such that 
\[
\lvert\{ F\in S_{n}(K)\mid F\in \varphi^{-1}(G),  G\in S_{n}(L)\} \lvert= c
\]
for every $F\in K$.
The quantity $c$ is also called the \emph{degree of the covering}.
\end{lem}
\begin{proof}
Let $G$ and $G'$ be $(n+1)$-up neighbours, i.e., there exists $\bar{G}\in S_{n+1}(L)$, s.t. $G,G'\in \partial \bar{G}$.
Due to Definition \ref{definition strong covering}, 
for every $F\in \varphi^{-1}(G)$, there exists $\bar{F}$, s.t. $\varphi(\bar{F})=\bar{G}$. Thus, there exists 
$F'\in \partial \bar{F}$ which is in the preimage of  $G'$ under the covering map $\varphi$.
Therefore, if  $G$ and $G'$ are $(n+1)$-up neighbours, then
$\lvert\{F\in S_{n}(K)\mid F\in \varphi^{-1}(G) \} \lvert= \lvert\{ F\in S_{n}(K)\mid F'\in \varphi^{-1}(G')\} \lvert$.
Since  $L$ is  an $(n+1)$-path connected simplicial complex, then for  arbitrary  $n$-faces $G$ and $G'$ of $L$,  there exists a sequence
$G=G_{0},G_{1},\ldots,G_{m}=G'$  of $n$-faces of the simplicial
complex $L$, such that any two neighbouring ones are $(n+1)$-up neighbours as well.
This gives us the same cardinal number of sets 
$\varphi^{-1}(G)$, $\varphi^{-1}(G')$, and $\varphi^{-1}(\bar{G})$.
\end{proof}
\begin{thm}
\label{Gustavson}
Let  $\varphi: K\rightarrow L$ be a strong covering, and let $\frac{w_{K}(\bar{F})}{w_{K}(F)}=\frac{w_{L}(\varphi(\bar{F}))}{w_{L}(\varphi(F))}$, for every pair
$\bar{F}, F$, such that $F\in \partial \bar{F}$. Then, $\varphi \mathcal{L}^{up}(L)=\mathcal{L}^{up}(K)\varphi$, i.e.,
$\s( \mathcal{L}^{up}(L))\subset \s(\mathcal{L}^{up}(K)\varphi)$.
\end{thm}

\begin{proof}
The proof follows directly from the equations below and the definition
of a strong covering.
\small
\begin{equation}
\label{left equiv}
\varphi \mathcal{L}^{up}(L)e_{[G]}=\sum_{\substack{\bar{G}\in S_{n+1}(L):\\ G',G\in \partial \bar{G}}}\sum_{\substack{F'\in S_{n}(K):\\ \varphi(F')=G'}}\sgn([F'],[G'])\sgn([G],\partial [\bar{G}])\sgn([G'],\partial [\bar{G}])\frac{w_{L}(\bar{G})}{w_{L}(G')}e_{[F']}.
\end{equation}
\begin{equation}
\label{right equiv}
\mathcal{L}^{up}(K)\varphi e_{[G]}=\sum_{\substack{F\in S_{n}(K):\\\varphi(F)=G}}\sum_{\substack{\bar{F}\in S_{n+1}(K):\\F,F'\in \partial \bar{F}}}\sgn([F],[G])\sgn([F],\partial [\bar{F}])\sgn([F'],\partial [\bar{F}])\frac{w_{K}(\bar{F})}{w_{K}(F')}e_{[F']}.
\end{equation}
\normalsize
First, we will prove 
\begin{equation}
\label{set equality}
\{F'\in S_{n}(K)\mid \exists \bar{G} \in S_{n+1}(L), \textrm{ s.t. }G, \varphi(F') \in \partial \bar{G} \}= \{F'\in S_{n}(K)\mid \exists \bar{F} \in S_{n+1}(K), \textrm{ s.t. } F, F' \in \partial \bar{F}, \textrm{ where } \varphi(F)=G \}.
\end{equation}
If a multiple of  $e_{[F']}$ occurs in the sum  in (\ref{right equiv}), then 
there exists $F\in S_{n}(K)$, such that $\varphi(F)=G$, and $F$ and $ F'$ are $(n+1)$-up neighbours.
Therefore, $\varphi(F')$ and $ G$ are $(n+1)$-up neighbours as well,
and $e_{[F']}$ is a summand in  (\ref{left equiv}).

On the other hand, if a multiple of  $e_{[F']}$ appears in  (\ref{left equiv}), then  $F'$ is an $n$-face of the simplicial complex $K$,  such that $\varphi(F')$ and $ G$ are $(n+1)$-up neighbours. Due to the properties of a strong covering, there must exist 
$F\in S_{n}(K)$, such that $\varphi(F)=G$, and $F$, $F'$ are
$(n+1)$-up neighbours; hence, $e_{[F']}$ is a summand in  (\ref{right equiv}).
Note that this claim will not hold if $\varphi$ is only a covering map
(see Example \ref{example coverings} and Figure \ref{cov05},  Figure
\ref{cov06} ). This proves (\ref{set equality}).

Due to   (\ref{equation signs})  we have 
\small
\begin{align*}
\sgn([F],[G])\sgn([F],\partial [\bar{F}])\sgn([F'],\partial [\bar{F}])=& \sgn([G],\partial [\bar{G}])\sgn([\bar{F}],[\bar{G}])\sgn([F'],\partial [\bar{F}])\\
=& \sgn([G],\partial [\bar{G}]) \sgn([G'],\partial [\bar{G}])\sgn([F'],[G']), 
\end{align*}
\normalsize
which makes (\ref{left equiv}) and (\ref{right equiv}) equal.
\end{proof}
As a consequence of  Theorem \ref{Gustavson} and Lemma \ref{lemma constant coverings} we obtain the following corollary.
\begin{coll}
\label{Corollary strong coverings}
If $\varphi: K\rightarrow L$ is a strong covering, then 
\begin{itemize}
\item[(i)] $\s(L^{up}_{n}(L))\subset\s(L^{up}_{n}(K))$
and
\item [(ii)] $\s(\Delta^{up}_{n}(L))\subset\s(\Delta^{up}_{n}(K))$.
\end{itemize} 
\end{coll}
\begin{proof}
First  we consider  the case $(i)$, when $w_{K}\equiv 1$.
Since $\varphi$ is a strong covering, then according to Theorem \ref{Gustavson} we have
$\s( \mathcal{L}^{up}(L))\subset \s(\mathcal{L}^{up}(K))$, where 
$w_{L}(G)=\sum_{F:\varphi(F)=G}w_{K}(F)$.
Due to Lemma \ref{lemma constant coverings} there exists a constant  $m=\lvert\{ F\in S_{n}(K)\mid F\in \varphi^{-1}(G)\} \lvert$, such that  $w_{L}(G)=m$ for every $G\in L$.
From the definition of the generalized Laplace operator, it follows that
$\mathcal{L}^{up}_{n}(L,w_{L})=L^{up}_{n}(L)$.
Case $(ii)$ is a direct consequence of   Theorem \ref{Gustavson}, Lemma \ref{lemma coverings}, and previous considerations.
\end{proof}
\begin{rem}
A similar theorem for coverings and the spectrum of the combinatorial Laplacian was proposed in \cite{Gustavson}.
In particular Gustavson claims that for a given covering map $\varphi: K\rightarrow L$ among simplicial complexes $K$ and $L$, the spectrum of the combinatorial Laplacian $L^{up}_{n}(L)$ is a subset of $L^{up}_{n}(L)$.
However, there are counterexamples to this claim.
For instance, the simplicial complex $\tilde{L}$ given in Figure
\ref{cov05} is a covering (according to Definition \ref{definition
  covering}) of the simplicial complex $L$ in Figure \ref{cov06}. The eigenvalues of 
$L^{up}_{0}(\tilde{L})$ are $0, 0.38, 1.38, 2.61, 3.61$, whereas the eigenvalues of $L^{up}_{0}(L)$ are $0, 1,3,4$.
The same definition \ref{definition covering} of a covering map  was used by Rotman in \cite{Rotman1}, who claimed  the  lifting lemma (Theorem 2.1.),  
which says  that every path in $L$ with  a base point $v$ can be uniquely lifted to a path in $K$ with a base point $\tilde{v}$, for every 
$\tilde{v}\in \varphi^{-1}(v)$.
However, the same pair of covering complexes can be used as a counterexample to this claim as well.
These counterexamples are eliminated if we consider strong coverings as Definition
\ref{definition strong covering}  instead of coverings.
\end{rem}

\begin{rem}
A combinatorial $k$-wedge $K_{1}\vee_{k}K_{2}$ has as a cover (not strong) $K_{1}\cup K_{2}$, thus Theorem \ref{Interlacing covering} holds, whereas Theorem \ref{Gustavson} does not.
\end{rem}

If a simplicial map $\varphi$ fails to be  a covering, then   $\varphi$ need not preserve the dimensionality, i.e., there may exist $n$-simplices in $K$, whose image under map $\varphi$ will be  $m$-dimensional, $m< n$.
Without  loss of generality assume 
$[\bar{F}]=[v_{0},\ldots,v_{n+1}]$, $[F_{i}]=[v_{0},\ldots,\hat{v_{i}},\ldots v_{n+1}]$,  $[F_{j}]=[v_{0},\ldots,\hat{v_{j}},\ldots v_{n+1}]$,  and 
 $\varphi(\bar{F})=\varphi(F_{i})=\varphi(F_{j})=G$ , i.e. $\varphi(v_{i})=\varphi(v_{j})$.
Then,
\small
\begin{align}
&\sum_{F\in \{F_{i},F_{j}\}}\sum_{\substack{\bar{F}\in S_{n+1}(K):\\ F,F'\in  \partial \bar{F}}} \sgn([F],[G])\sgn([F],\partial [\bar{F}])[F',G']\sgn([F]',\partial [\bar{F}])\frac{w_{K}(\bar{F})}{w_{L}(G')}e_{[G']}\\
&=\frac{w_{K}(\bar{F})}{w_{L}(G)}\sgn([F_{i}],[G])(-1)^{i}\sgn([F_{j}],[G])(-1)^{j}e_{[G]}+  \frac{w_{K}(\bar{F})}{w_{L}(G)}\sgn([F_{i}],[G])(-1)^{i}\sgn([F_{i}],[G])(-1)^{i}e_{[G]}\\
&+ \frac{w_{K}(\bar{F})}{w_{L}(G)}\sgn([F_{i}],[G])(-1)^{i}\sgn([F_{j}],[G])(-1)^{j}e_{[G]}+ \frac{w_{K}(\bar{F})}{w_{L}(G)}\sgn([F_{j}],[G])(-1)^{j}\sgn([F_{j}],[G])(-1)^{j}e_{[G]}\\
&=\frac{w_{K}(\bar{F})}{w_{L}(G)}(\sgn([F_{i}],[G])(-1)^{i}(-1)^{j}\sgn([F_{i}],[G])(-1)^{j-i+1}e_{[G]}+ e_{[G]})\\
&+ \frac{w_{K}(\bar{F})}{w_{L}(G)}(\sgn([F_{j}],[G])(-1)^{j}(-1)^{i}\sgn([F_{j}],[G])(-1)^{j-i+1}e_{[G]}+ e_{[G]})\\
&=0\label{equality star}
\end{align}
\normalsize
where $G'=\varphi(F')$.
Therefore, $\mathcal{L}^{up}(L)e_{[G]}=\varphi^{*} \mathcal{L}^{up}(K)\varphi e_{[G]}$ and the diagram (\ref{diagram simplicial map}) commutes for both choices of the weight function $w_{L}$, i.e., (\ref{rule (i)}) and (\ref{rule (ii)}).
Furthermore, the following holds.
\begin{thm}
\label{theorem intercing simplicial map rule(ii)}
Let $\varphi: K\rightarrow L$ be a simplicial map, and $w_{K}$ and
$w_{L}$ be weight functions of the simplicial complexes $K$ and $L$
respectively, such that (\ref{rule (ii)}) is satisfied. Let
$\lambda_{1},\ldots,\lambda_{N_{K}}$ and
$\theta_{1},\ldots,\theta_{N_{L}}$ be the eigenvalues of 
$\mathcal{L}^{up}_{n}(K,w_{K})$ and $\mathcal{L}^{up}_{n}(L,w_{L})$, respectively, ordered increasingly.
Then,
\begin{equation}
\lambda_{k}\leq \theta_{k}\leq \lambda_{k+N_{K}-N_{L}},
\end{equation} 
whith $\lambda_{N_{K}+1}=\ldots=\lambda_{2N_{K}-N_{L}}=N$, where $N$ is the number of vertices of  $K$.
\end{thm}
\begin{proof}
Follows directly from (\ref{equality star}), the  fact that $(g,g)=(\varphi g,\varphi g)$, and Theorem \ref{Interlacing covering}.
\end{proof}
 On the other hand, for the weight function $w_{L}$ as defined in
 (\ref{rule (i)}), we need certain modifications  for the interlacing theorem to hold.
 
 First, we motivate the choice of weight function $w_{L}$  in (\ref{rule (i)}).
 Let  $(K,w_{K})$ be a simplicial complex whose weights are normalized, i.e. 
$w_{K}(F)=\sum_{\bar{F}: F\in \bar{F}}w_{K}(\bar{F})$, and let 
 $\varphi: K\rightarrow L$ be a simplicial map. 
Assume   $\bar{F}_{1}, \ldots, \bar{F}_{k}$ are $(n+1)$-faces of $K$
whose $\varphi$-images are $n$-dimensional, i.e.
$\varphi(\bar{F}_{i})=G_{i}\in S_{n}(L)$, and   assume  (\ref{rule
  (ii)})  holds. Note that  we are now dealing with simplicial
complexes with loops, i.e., $L$ is such a complex and 
$\varphi(\bar{F}_{i})$ are loops; thus, 
$\deg G=\sum_{\substack{\bar{G}\in S_{n+1}(L):\\ G\in \partial \bar{G} }} w_{L}(\bar{G})+\sum_{\substack{\bar{F_{i}\in S_{n+1}(K)}: \\G\in \partial \varphi(\bar{F}_{i})}}w_{K}(\bar{F}_{i})$, and 
\begin{align*}
w_{L}(G_{i})=&\sum_{\substack{F\in S_{n}(K):\\ \varphi(F)=G}}w_{K}(F)\\
=&\sum_{\substack{ \bar{F} \in S_{n+1}(K), \bar{F}\neq \bar{F_{i}}:\\ F\in \partial\bar{F}}}w_{K}(\bar{F})+ 2w_{K}(\bar{F}_{i})\\
=&\sum_{\substack{ \bar{F} \in S_{n+1}(K):\\ F\in \partial\bar{F}}}w_{K}(\bar{F})+ w_{K}(\bar{F}_{i})\\
=& \deg G_{i} + w_{K}(\bar{F}_{i}).
\end{align*}
Therefore the weight function $w_{L}$  is not normalized.
However, if we adopt definition (\ref{rule (i)}) for $w_{L}$, we get the desired equality
$w_{L}(G)=\deg G$; hence, if $\mathcal{L}^{up}_{n}(K,w_{K})$ is  the normalized Laplacian, then 
$\mathcal{L}^{up}_{n}(L,w_{L})$ is  the normalized Laplacian as well.

  Let 
$\mathcal{W}$ be, as before,  a vector  space spanned by $\{\sgn([F_{ji}],[G_{i}])e_{[F_{ji}]}-\sgn([F_{(i+1)j}],[G_{i}])e_{[F_{(i+1)j}]}\mid \bigcup_{j} F_{ji}=\varphi^{-1}(G_{i}), \textrm{ and  }   \bigcup_{i} G_{i}=S_{n}( L)\}$.
Let $\phi: C^{n}(L)\rightarrow C^{n}(L)$ be an operator such that 
$\phi e_{[G]}=\sum_{\substack{\bar{F}\in S_{n+1}: \varphi(\bar{F})=G\\ }} \frac{w_{K}(\bar{F})}{w_{L}(G)}e_{[G]}$,
 and let $\mathcal{Z}$ be a vector space of dimension $z$ spanned by vectors $\{e_{[G]}\mid G\in S_{n}(L), \exists \bar{F}\in S_{n+1}(K) \textrm{ s.t. } \varphi(\bar{F})=G\}$.
Assume that $\varphi: K\rightarrow L$ is a simplicial map, and $w_{K}$ and $w_{L}$ are weight functions satisfying \ref{rule (i)}.
Then, we have
\begin{align}
\theta_{k}=& \max_{\mathcal{V}_{k-1}}\min_{g\perp \mathcal{V}_{k-1}}\frac{(\mathcal{L}^{up}(L)g,g)}{(g,g)}\\
=& \max_{\mathcal{V}_{k-1}}\min_{g\perp \mathcal{V}_{k-1}}\frac{(\varphi^{*} \mathcal{L}^{up}(K)\varphi g ,g)}{(g,g)}\\
=& \max_{\mathcal{V}_{k-1}}\min_{g\perp \mathcal{V}_{k-1}}\frac{( \mathcal{L}^{up}(K)\varphi g ,\varphi g)}{(\varphi g,\varphi g) - (\phi g, g)}\label{ineq M04}\\
\geq & \max_{\mathcal{V}_{k-1}}\min_{g\perp \mathcal{V}_{k-1}}\frac{( \mathcal{L}^{up}(K)\varphi g ,\varphi g)}{(\varphi g,\varphi g) }\\
= & \max_{\mathcal{V}_{k-1}}\min_{f\perp \mathcal{V}_{k-1}, f\perp \mathcal{W}}\frac{( \mathcal{L}^{up}(K)f , f)}{(f,f) }\\
\geq & \max_{\mathcal{V}_{k-1}}\min_{g\perp \mathcal{V}_{k-1}} \frac{( \mathcal{L}^{up}(K)f , f)}{(f,f) }\\
\geq & \lambda_{k}.
\end{align}


On the other hand, the upper interlacing inequality follows from  (\ref{ineq M04}). 
\begin{align}
\theta_{k}  \leq & \max_{\mathcal{V}_{k-1}}\min_{g\perp \mathcal{V}_{k-1}}\frac{( \mathcal{L}^{up}(K)\varphi g ,\varphi g)}{(\varphi g,\varphi g) - (\phi g, g)}\\
\leq & \max_{\mathcal{V}_{k-1}}\min_{g\perp \mathcal{V}_{k-1}, g\perp \mathcal{Z}}\frac{( \mathcal{L}^{up}(K)\varphi g ,\varphi g)}{(\varphi g,\varphi g)}\\
\leq & \max_{\mathcal{V}_{k+z-1}}\min_{g\perp \mathcal{V}_{k+z-1}, }\frac{( \mathcal{L}^{up}(K)\varphi g ,\varphi g)}{(\varphi g,\varphi g)}\\
= & \max_{\mathcal{V}_{k+z-1}}\min_{f\perp \mathcal{V}_{k+z-1}, f\perp \mathcal{W} }\frac{( \mathcal{L}^{up}(K)f ,f)}{(f,f)}\\
\leq & \max_{\mathcal{V}_{k+z+N_{K}-N_{L}-1}}\min_{f\perp \mathcal{V}_{k+z+N_{K}-N_{L}-1} }\frac{( \mathcal{L}^{up}(K)f ,f)}{(f,f)}\\
\leq & \lambda_{k+N_{K}-N_{L}+z}
\end{align}
The inequalities here are derived analogously to the ones in Theorem \ref{Interlacing covering}.
We assemble our results into the following theorem.
\begin{thm}
\label{Theorem interlacing 2nd choice of weights}
Let $(K,w_{K})$ and $(L,w_{L})$ be weighted simplicial complexes and $\varphi:K\rightarrow L$ a simplicial map, such that 
(\ref{rule (i)}) holds.
Let $\mathcal{Z}$ be the  vector space of dimension $z$ with  basis  $\{e_{[G]}\mid \exists \bar{F}\in S_{n+1}(K) \textrm{ s.t. } \varphi(\bar{F})=G\}$,  and let
$\mathcal{W}$ be the vector space  spanned by $\{\sgn([F_{ji}],[G_{i}])e_{[F_{ji}]}-\sgn([F_{(i+1)j}],[G_{i}])e_{[F_{(i+1)j}]}\mid \bigcup_{j} F_{ji}=\varphi^{-1}(G_{i}), \textrm{ and  }   \bigcup_{i} G_{i}=S_{n}( L)\}$.  Assume 
$\lambda_{1}\leq \lambda_{2}\leq \ldots\leq \lambda_{N_{K}}$ and
$\theta_{1}\leq \theta_{2}\leq \ldots\leq \theta_{N_{K}}$ are the eigenvalues of $\mathcal{L}^{up}_{n}(K)$ and  $\mathcal{L}^{up}_{n}(L)$, respectively.  Then, 
\begin{equation}
\lambda_{k} \leq \theta_{k}\leq   \lambda_{k+z+N_{K}-N_{L}}
\end{equation}
\end{thm}
\begin{rem}
In \cite{Butler} Section 3,  Butler obtains interlacing effects of  \emph{ weak coverings}.
Given weighted graphs $(G,w_{G})$  with normalized weights and $(H,w_{H})$, then  
$\pi: G\rightarrow H$ is a weak covering map, if it   maps vertices to
vertices. The weight functions of  $H$ and $G$ satisfy the following.
Let $(x,y)$ be an edge in $H$, then $w_{H}(x,y)=\sum_{\substack{u\in \pi^{-1}(x),\\ v\in \pi^{-1}(y)}}w_{G}(u,v)$, and let $x$ be a vertex in $H$, then $\deg_{H} x=\sum_{v\in \pi^{-1}(x)} \deg_{G} v$.
In our terminology, a "weak cover" is nothing but a simplicial map
defined on graphs such that the weights obey 
(\ref{rule (ii)}).
However, if $L$ happens to have loops, stemming from an edge collapsing into a vertex, the Laplace operator $\mathcal{L}^{up}_{0}(H,w_{H})$ will not be  normalized. For instance, let $G$ be a hollow traingle, with edges 
$(1,2), (2,3)$ and $(1,3)$ with weights $1$ on the edges, and $2$ on vertices; and let $H$ be a graph on one edge $(1',2')$ and a loop $(2',2')$.
Then  the map $\pi: G\rightarrow H$  with 
$\pi(1)=1'$ and $\pi(2)=\pi(3)=2'$ is a weak covering map.
Thus,  $w_{H}(1',2')=2$, $\deg _{H}(1')=2$, and  $\deg _{H}(2')=4\neq
\sum_{(u,v):\pi(u)=2' }w_{G}(u,v)=3$, which is clearly in
contradiction with the weights required for the normalized graph Laplacian.
However, the interlacing theorem Butler obtains is accurate, for weights defined as
$w_{H}(x)=\sum_{u: \pi(u)=x} w_{K}(v)$, but the resulting Laplacian $\mathcal{L}^{up}_{0}(H,w_{H})$ will not be normalized!
\end{rem}
\section{Collapsing, Contracting and Interlacing}
\label{section 4}
\label{Collapsing, Contracting}
In this section we analyze the effect of collapsing and contraction on the eigenvalues of the Laplace operator.
\begin{defn}
Let $K$ be a simplicial complex and $(\bar{F}, F)$ pair of faces, such that 
$F\in \partial\bar{F}$, and $F$  is not a facet of any other face in $K$.
The face $\bar{F}$ is  called a \emph{free face}, and a simplicial complex $K'$ obtained from $K$ by deleting 
$\bar{F}$ and $F$ is called an \emph{elementary collapse} of $K$. A sequence of elementary collapses is called a \emph{collapse}.
A simplicial map $\varphi:K\rightarrow K'$, corresponding to this
operation, is also called an elementary collapse.
\end{defn}
\begin{defn}
\label{definition contraction}
Let $K$ be a simplicial complex of dimension $n+1$, and let 
$\bar{F}$ be an $(n+1)$-face, such that  only two of its facets $F,F'$, are incident to $(n+1)$-simplices other than $\bar{F}$.
The simplicial complex $K'$ obtained from $K$ by deleting the face $\bar{F}$ and identifying $F$ and $F'$ is called an \emph{elementary contraction} 
of  $K$.  A sequence of elementary contractions is called a \emph{contraction}.
A simplicial map $\varphi:K\rightarrow K'$, corresponding to this
operation, is also called an elementary contraction.
\end{defn}
\begin{rem}
Assume $F$ and $F'$ are $n$-faces of the $(n+1)$-face $\bar{F}$, which
are identified under the elementary contraction 
$\varphi:K\rightarrow K'$. If  $F'$ has no  $(n+1)$-up neighbours, other than facets of $\bar{F}$, then this elementary contraction can be treated as a composition of two elementary  collapses.  Thus, we will not analyse this case separately. 
\end{rem}
Let $\varphi:K\rightarrow K'$ be an elementary contraction, which identifies two $n$-faces $F_{1}$ and $F_{1}'$, which are $(n+1)$-up neighbours. In other words, there exist  vertices $v\in F_{1}$ and $v'\in F_{1}'$, such that 
$\varphi(v)=\varphi(v')$, and $\varphi $ is injective on all other vertices of $K$.
In what follows we will  distinguish among two types of elementary
contractions: 
\begin{itemize}
\item[$(i)$]  there exist $2m$  $(n+1)$-faces 
 $\bar{F}_{2}$, $\bar{F}'_{2}, \ldots, \bar{F}_{m+1}, \bar{F}'_{m+1}$ of  $K$, such that $\varphi(\bar{F}_{i})=\varphi(\bar{F}'_{i})$,
 for all $2\leq i \leq m+1$, and $\varphi$ is injective on the remaining $(n+1)$-faces of $K$, \\
 or
 \item[$(ii)$] $\varphi$ is injective on $S_{n+1}(K)$.
\end{itemize}

\begin{exmp}
In Figures \ref{collapse02} and  \ref{collapse03}, elementary
collapses of the simplicial complex 
$K$ (Figure \ref{SC}) are presented, in particular the collapse of  an
edge and a triangle, respectively.
Figures \ref{contraction01} and  \ref{contraction02} represent two main types of elementary contractions, type $(i)$ and $(ii)$, respectively.
\end{exmp}
\begin{figure}[h!tp]
  \begin{center}
\includegraphics[scale=0.45]{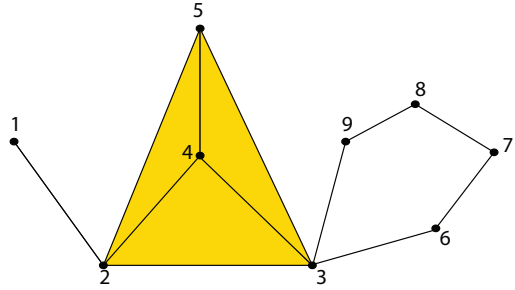}
  \end{center}
  \caption{Simplicial complex $K$}
  \label{SC}
\end{figure} 

 \begin{figure}[h!tp]
  \begin{center}
 \subfigure[raggedright,scriptsize][ ]{\label{collapse02}\includegraphics[scale=0.40]{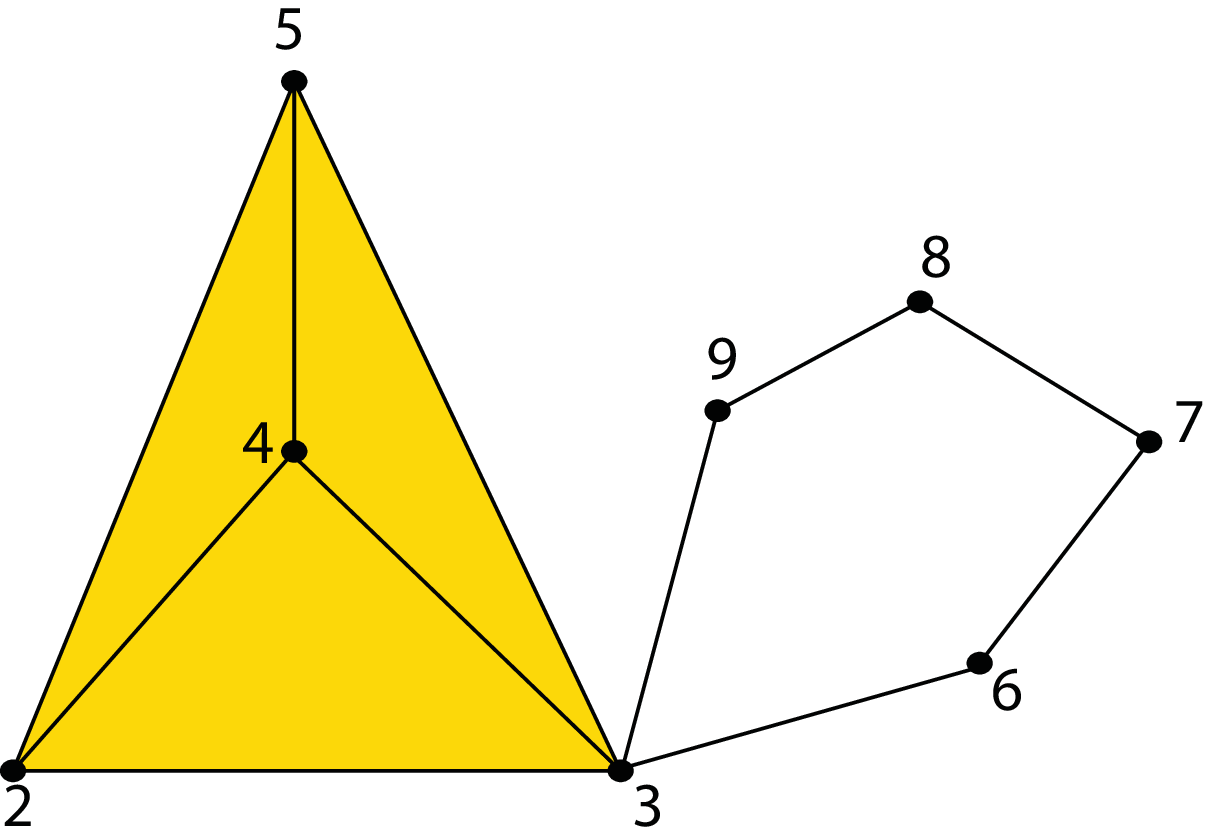}} \qquad\qquad
    \subfigure[raggedright,scriptsize][  ]{\label{collapse03}\includegraphics[scale=0.40]{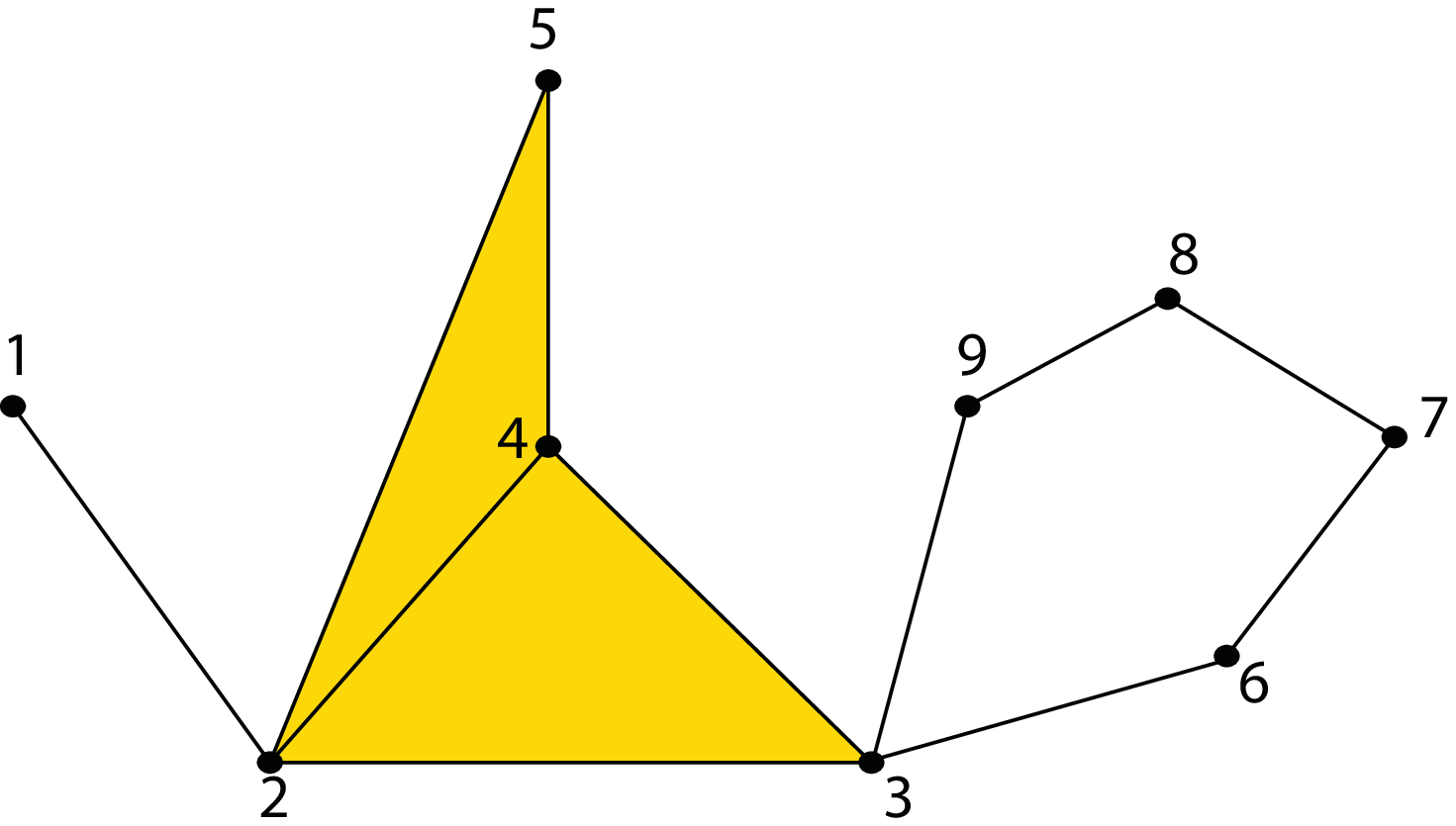}}
  \end{center}
  \caption{Elementary collapses of simplicial complex $K$ }
  \label{figure collapses}
\end{figure} 

 \begin{figure}[h!tp]
  \begin{center}
 \subfigure[raggedright,scriptsize][ ]{\label{contraction01}\includegraphics[scale=0.4]{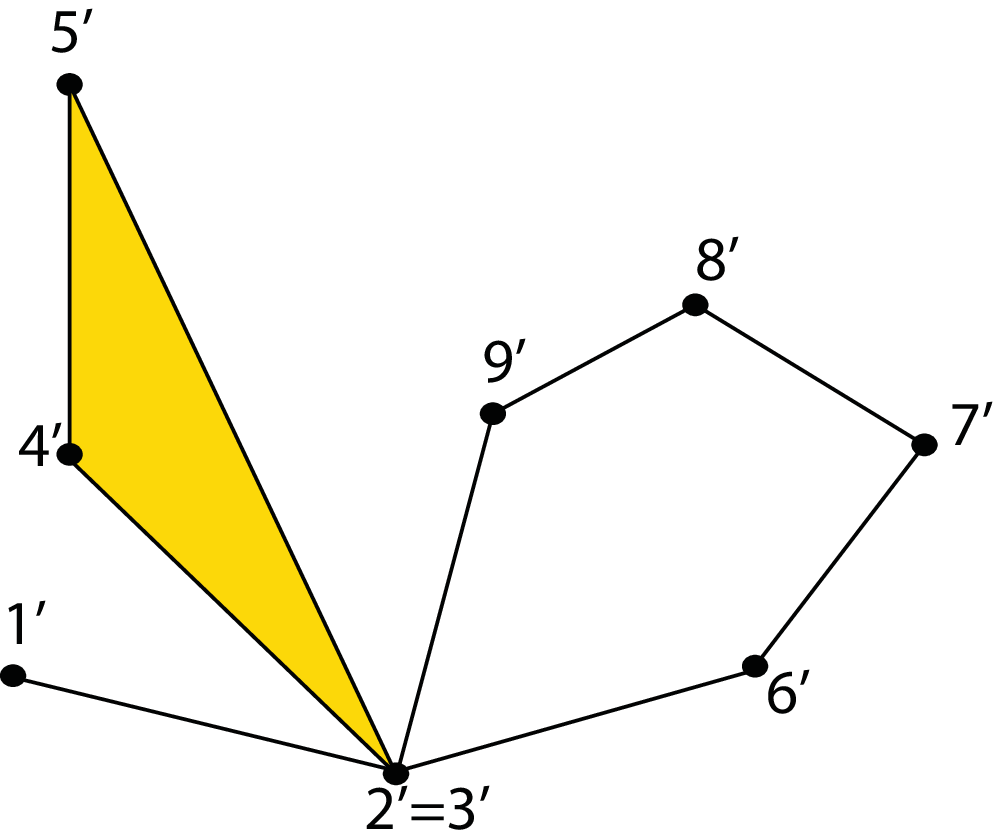}} \qquad\qquad
    \subfigure[raggedright,scriptsize][  ]{\label{contraction02}\includegraphics[scale=0.4]{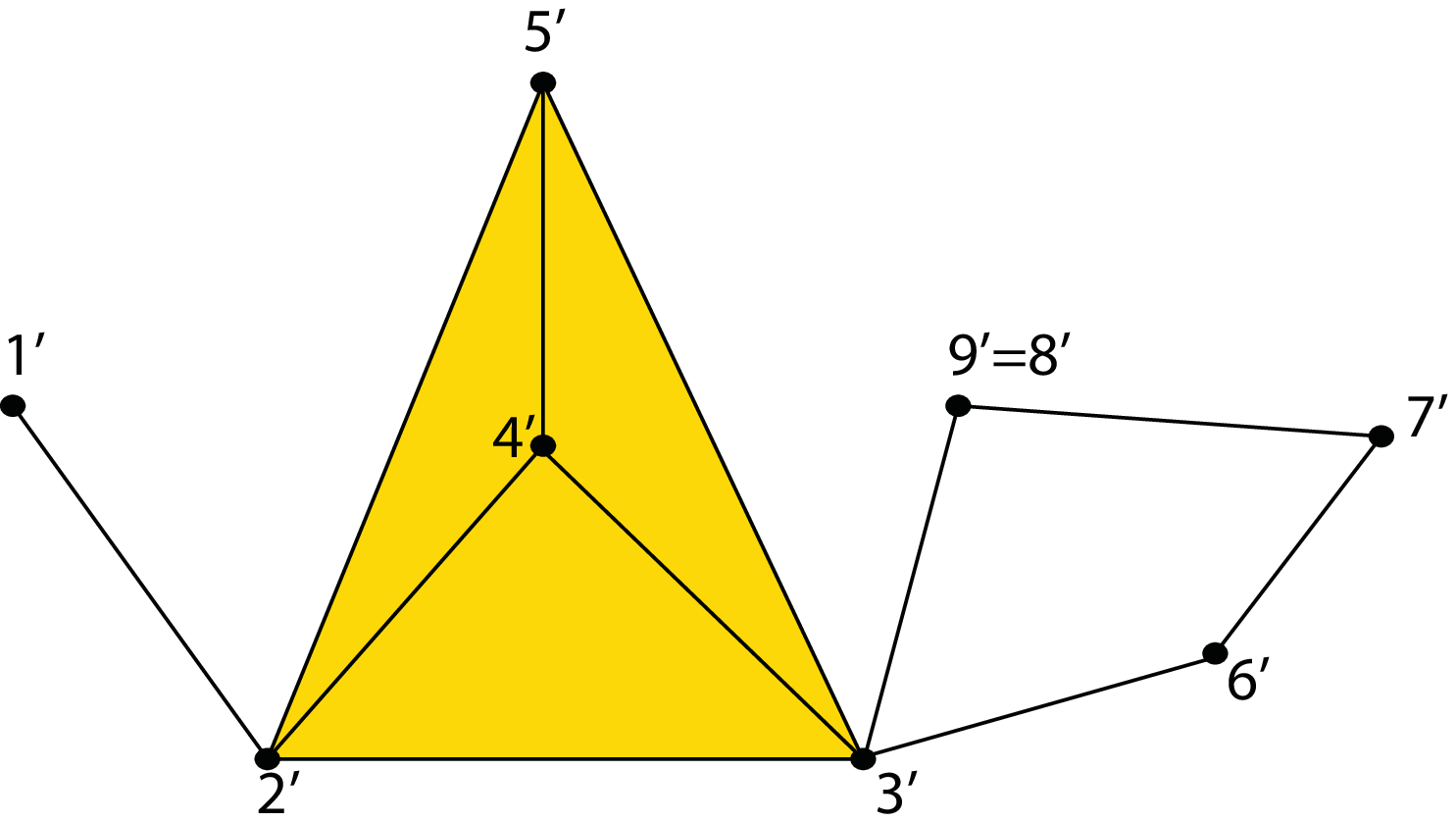}}
  \end{center}
  \caption{Elementary contractions of simplicial complex $K$ }
  \label{figure contrations}
\end{figure} 

Note that both these operations, collapse and contraction, are simplicial maps. However, the interlacing theorems proven Section \ref{section 3} are valid 
for  specific choices of weights on $K'$, see (\ref{rule (i)}) and (\ref{rule (ii)}).
For the combinatorial Laplacian $L^{up}_{n}(K)=\mathcal{L}^{up}_{n}(K,w_{K})$, where $w_{K}\equiv 1$, one might want to consider the case 
of interlacing eigenvalues of $L^{up}_{n}(K)$ ($\Delta^{up}_{n}(K)$) and $L^{up}_{n}(K')$ ($\Delta^{up}_{n}(K')$), where $K'$ is either a collapse or a contraction of $K$. 
These cases are resolved in the following theorems.
\begin{thm}
Let  $\varphi:K\rightarrow K'$  be an  elementary contraction, and let
$\lambda_{1}\leq \ldots\lambda_{N_{K}}$, $\theta_{1}\leq
\ldots\theta_{N_{K'}}$ be the eigenvalues of 
$L^{up}_{n}(K)$ and $L^{up}_{n}(K')$($\Delta^{up}_{n}(K)$ and $\Delta^{up}_{n}(K')$), respectively, then 
\begin{itemize}
\item[(i)] \begin{equation}
\lambda_{k-m(n+2)}\leq  \theta_{k}\leq \lambda_{k+N_{K}-N_{K'}+m(n+2)},
\end{equation}
if the contraction $\varphi$ is of type $(i)$, or 
\item [(ii)]
\begin{equation}
\lambda_{k}\leq  \theta_{k}\leq \lambda_{k+n+2},
\end{equation}
if  $\varphi$ is of type $(ii)$,
\end{itemize}
where $ \lambda_{N_{K}+1}=\ldots=\lambda_{2N_{K}-N_{K'}+m(n+1)}=N$,  $\lambda_{0}=\lambda_{-1}=\ldots=\lambda_{-n-1}=0$,  and $N$ is the number of vertices of  $K$.
\end{thm}
\begin{proof}
Assume  that the weight functions $w_{K}$ and $w_{K'}$ are identically
equal to $1$, i.e.,  we are dealing with the combinatorial Laplacian $L^{up}_{n}$.
Let $\bar{F}_{1}$,  $F_{1}$, and $F'_{1}$ be 
the faces of  $K$, such that the elementary contraction $\varphi$ is given by identification of $F_{1}$ and $F'_{1}$.
In the proof we will distinguish among two, already mentioned,  cases of elementary collapses $\varphi$.
Denote by $\bar{F}_{2}$, $\bar{F}'_{2}, \ldots, \bar{F}_{m+1},
\bar{F}'_{m+1}$ the $(n+1)$-faces of  $K$ that are identified 
under $\varphi$, if $\varphi$ is of type $(i)$.   Let  $\psi: C^{n}(K',\rb)\rightarrow C^{n}(K',\rb)$ be a map, such that 
\begin{equation}
\psi e_{[G]}=\sum_{G'\in \partial \bar{G}_{i}}w_{K}(\bar{F}_{i})\sgn([G],\partial [\bar{G}_{i}] ) \sgn([G'],\partial [\bar{G}_{i}] ) \frac{1}{w_{K'}(G')}e_{[G']},
\end{equation}
if $\varphi$ is of type $(i)$, $G\in \partial \varphi(\bar{F}_{i})$, where $\bar{G}_{i}=\varphi(\bar{F}_{i})$, for all $2\leq i \leq m+1$, and 
\begin{equation}
\psi e_{[G]}=0,
\end{equation}
for other choices of $G\in K'$ and for $\varphi$ of type $(ii)$.

According to (\ref{contraction equation}) and the considerations for simplicial maps, we have
\begin{equation}
\varphi^{*}\mathcal{L}^{up}_{n}(K)\varphi e_{G}=\sum_{\substack{\bar{F}\in S_{n+1}(K):\\G\in \partial\varphi(\bar{F}) }} w_{K}(\bar{F})\sum_{G'\in \partial \bar{G}} \sgn([G],\partial[\bar{G}]) \sgn([G'],\partial[\bar{G}])\frac{e_{[G']}}{w_{K'}(G')}.
\end{equation}
Therefore,
\small
\begin{equation}
\varphi^{*}L^{up}_{n}(K)\varphi e_{[G]}= L^{up}_{n}(K')e_{[G]} + w_{K}(\bar{F}_{i})\sum_{G'\in \partial \bar{G}_{i}} \sgn([G],\partial[\bar{G}_{i}]) \sgn([G'],\partial[\bar{G}_{i}])\frac{e_{[G']}}{w_{K'}(G')},
\end{equation}
\normalsize
if $\varphi$   is of type $(i)$ and $G\in \partial \varphi(\bar{F}_{i})$,  $2\leq i \leq m+1$, and 
\begin{equation}
\varphi^{*}L^{up}_{n}(K)\varphi e_{[G]}=L^{up}_{n}(K')e_{[G]} ,
\end{equation}
otherwise.
Thus, $L^{up}_{n}(K')=\varphi^{*}L^{up}_{n}(K)\varphi- \psi $.
As for  $(\varphi g, \varphi g)$ we have the following equalities
\begin{align}
(\varphi g,\varphi g) =& \sum_{G\in S_{n}(K')}\sum_{\substack{F\in S_{n}(K):\\ \varphi(F)=G}}(\sgn([F],[G])g(\varphi F))^{2}w_{K}(F)\\
=& \sum_{G\in S_{n}(K')} g([G])^{2} \sum_{\substack{F\in S_{n}(K):\\ \varphi(F)=G}}w_{K}(F)\\
=&(g,g)+  (\phi_{1}g,g) +(\phi_{2}g,g),
\end{align}
where 
\begin{equation}
\phi_{1}e_{[G]}=\left\{\begin{array}{ll}
 w_{K}(F)e_{[G]} & \textrm{ if } \varphi \textrm{ is of type } (i) \textrm{, and } \varphi(F)=G\in \partial \varphi(\bar{F}_{i}), 2\leq i \leq m+1,\\
0 & \textrm{ otherwise, }
\end{array}
\right.
\end{equation}
and
\begin{equation}
\phi_{2}e_{[G]}=\left\{\begin{array}{ll}
w_{K}(F_{1})e_{[G]} & \textrm{ if } F_{1}\in \varphi^{-1}(G)\textrm{, and } \varphi \textrm{ is of type }(ii),\\
0& \textrm{ otherwise.}
\end{array}
\right.
\end{equation}
It  is now straightforward  to deduce the interlacing inequalities. Namely, 
\begin{align}
\theta_{k}=& \min_{\mathcal{V}_{N_{K'}-k}}\max_{g\perp \mathcal{V}_{N_{K'}-k}}\frac{(\mathcal{L}^{up}(K')g,g)}{(g,g)}\\
=& \min_{\mathcal{V}_{N_{K'}-k}}\max_{g\perp \mathcal{V}_{N_{K'}-k}}\frac{(\varphi^{*} \mathcal{L}^{up}(K)\varphi g ,g)-(\psi g,g)}{(g,g)}\\
\geq & \min_{\mathcal{V}_{N_{K'}-k}}\max_{g\perp \mathcal{V}_{N_{K'}-k}, g\perp \mathcal{Y}}\frac{(\varphi^{*} \mathcal{L}^{up}(K)\varphi g ,g)}{(g,g)}\label{ineq C01}\\
= & \min_{\mathcal{V}_{N_{K'}-k}}\max_{g\perp \mathcal{V}_{N_{K'}-k}, g\perp \mathcal{Y}}\frac{( \mathcal{L}^{up}(K)\varphi g ,\varphi g)}{(\varphi g,\varphi g) - (\phi_{1} g, g)-(\phi_{2} g, g)}\label{ineq C02}\\
\geq & \min_{\mathcal{V}_{N_{K'}-k}}\max_{g\perp \mathcal{V}_{N_{K'}-k}, g\perp \mathcal{Y}}\frac{( \mathcal{L}^{up}(K)\varphi g ,\varphi g)}{(\varphi g,\varphi g)}\label{ineq C03}\\
\geq & \min_{\mathcal{V}_{N_{K'}-k+y}}\max_{g\perp \mathcal{V}_{N_{K'}-k+y}}\frac{( \mathcal{L}^{up}(K)\varphi g ,\varphi g)}{(\varphi g,\varphi g)}\label{ineq C04}\\
= & \min_{\mathcal{V}_{N_{K'}-k+y}}\max_{f\perp \mathcal{V}_{N_{K'}-k+y},f\perp \mathcal{W}}\frac{( \mathcal{L}^{up}(K)f,f)}{(f,f)}\label{ineq C05}\\
\geq & \min_{\mathcal{V}_{N_{K'}-k+y}}\max_{f\perp \mathcal{V}_{N_{K}-k+y}}\frac{( \mathcal{L}^{up}(K)f,f)}{(f,f)}\label{ineq C06}\\
\geq & \min_{\mathcal{V}_{N_{K}-k+y}}\max_{f\perp \mathcal{V}_{N_{K}-k+y}}\frac{( \mathcal{L}^{up}(K)f,f)}{(f,f)}\\
\geq &\lambda_{k-y}.
\end{align}
In  inequality (\ref{ineq C01}), $\mathcal{Y}$ denotes the subspace of $C^{n}(K',\rb)$ on which $(\psi g,g)=0$ and the dimension of this space ($\dim \mathcal{Y}=y$)  is
 at most  $(n+2)m$, if $\varphi$ is of type $(i)$, and   $y=0$, if $\varphi$ is an elementary contraction of type $(ii)$.
The vector space $\mathcal{W}$ in (\ref{ineq C06}) is generated by 
$\{\sgn([F_{ji}],[G_{i}])e_{[F_{ji}]}-\sgn([F_{(i+1)j}],[G_{i}])e_{[F_{(i+1)j}]}\mid \bigcup_{j} F_{ji}=\varphi^{-1}(G_{i}), \textrm{ and  }   \bigcup_{i} G_{i}=S_{n}( K')\}$,
and of dimension $\dim \mathcal{W}=N_{K}-N_{K'}$.
Note that the quantity $N_{K}-N_{K'}$ equals $n+1$, if $\varphi$  is
an elementary collapse  of type $(ii)$.

Similarly, the upper interlacing inequality follows from  (\ref{ineq M04}), i.e.,
\begin{align}
\theta_{k}  \leq & \max_{\mathcal{V}_{k-1}}\min_{g\perp \mathcal{V}_{k-1}}\frac{( \mathcal{L}^{up}(K)\varphi g ,\varphi g)-(\psi g,g)}{(\varphi g,\varphi g) - (\phi_{1} g, g)- (\phi_{2} g, g)}\\
\leq & \max_{\mathcal{V}_{k-1}}\min_{g\perp \mathcal{V}_{k-1}}\frac{( \mathcal{L}^{up}(K)\varphi g ,\varphi g)}{(\varphi g,\varphi g) - (\phi_{1} g, g)- (\phi_{2} g, g)}\\
\leq & \max_{\mathcal{V}_{k-1}}\min_{g\perp \mathcal{V}_{k-1}, g\perp \mathcal{Z}}\frac{( \mathcal{L}^{up}(K)\varphi g ,\varphi g)}{(\varphi g,\varphi g)}\label{ineq C07}\\
\leq & \max_{\mathcal{V}_{k+z-1}}\min_{g\perp \mathcal{V}_{k+z-1}, }\frac{( \mathcal{L}^{up}(K)\varphi g ,\varphi g)}{(\varphi g,\varphi g)}\\
= & \max_{\mathcal{V}_{k+z-1}}\min_{f\perp \mathcal{V}_{k+z-1}, f\perp \mathcal{W} }\frac{( \mathcal{L}^{up}(K)f ,f)}{(f,f)}\\
\leq & \max_{\mathcal{V}_{k+z+N_{K}-N_{L}-1}}\min_{f\perp \mathcal{V}_{k+z+N_{K}-N_{L}-1} }\frac{( \mathcal{L}^{up}(K)f ,f)}{(f,f)}\\
\leq & \lambda_{k+N_{K}-N_{L}+z}.
\end{align}
The vector space $\mathcal{Z}$ appearing in inequality (\ref{ineq C07}) is   $\mathcal{Z}=\{g\in C^{n}(K',\rb)\mid  (\phi_{1} g, g)+(\phi_{2} g, g)=0\}$, and 
the dimension $z=\dim \mathcal{Z}$, is equal to  $1$,  if $\varphi$ is of type $(ii)$, and  $m(n+2)$ otherwise.
Thus, in the case $(i)$ we have the following interlacing inequalities:
\begin{equation}
\lambda_{k-m(n+2)}\leq  \theta_{k}\leq \lambda_{k+N_{K}-N_{K'}+m(n+2)}.
\end{equation}
And the case $(ii)$ results in 
\begin{equation}
\lambda_{k}\leq  \theta_{k}\leq \lambda_{k+N_{K}-N_{K'}+1}.
\end{equation}
A very similar method can be used to prove inequalities for the normalized Laplacian,
with  the  difference in the definition of weight functions and maps $\phi_{1}$, $\phi_{2}$.
In this case, the weight functions  $w_{K}$ and $w_{K'}$ have value $1$ on all faces of dimension $n+1$,  whereas
$w_{K}(F)=\deg F$, and $w_{K'}(G)=\deg G$, for all $n$-faces  $F$, $G$, of $K$ and $K'$, respectively.
The maps $\phi_{1},\phi_{2}:C^{n}(K',\rb)\rightarrow C^{n}(K',\rb)$ are given by
\begin{equation}
\phi_{1}e_{[G]}=\left\{\begin{array}{ll}
 w_{K}(\bar{F}_{i})e_{[G]} & \textrm{ if } \varphi \textrm{ is of type } (i) \textrm{, and }G\in \partial \bar{G}_{i}, 2\leq i \leq m+1, \\
0 & \textrm{ otherwise, }
\end{array}
\right.
\end{equation}
and 
\begin{equation}
\phi_{2}e_{[G]}=\left\{\begin{array}{ll}
2w_{K}(\bar{F}_{1})e_{[G]} & \textrm{ if }  G=\varphi(F_{1}) ,\\
0& \textrm{ otherwise.}
\end{array}
\right.
\end{equation}
However, the remainder of the proof is exactly the same as in the case
of the combinatorial Laplacian $L_{n}^{up}$, and so are the dimensions
of vector spaces $\mathcal{Z}$, $\mathcal{Y}$ and $\mathcal{W}$; thus the same interlacing inequalities hold.
\end{proof}
We now briefly discuss interlacing inequalities for elementary collapses.
\begin{thm}
Let $\varphi:K\rightarrow K'$ be an elementary collapse, and let 
$\lambda_{1}\leq \ldots\lambda_{N_{K}}$ and $\theta_{1}\leq
\ldots\theta_{N_{K'}}$ be the eigenvalues of 
$L^{up}_{n}(K)$ and $L^{up}_{n}(K')$($\Delta^{up}_{n}(K)$ and $\Delta^{up}_{n}(K')$), respectively, then 
 \begin{equation}
\lambda_{k}\leq  \theta_{k}\leq \lambda_{k+n+3},
\end{equation}
where $N= \lambda_{N_{K}+1}=\ldots=\lambda_{N_{K}+n+3}$, and $N$ is the number of vertices of  $K$.
\end{thm}
\begin{proof}
A direct consequence of Theorem \ref{main theorem} and Corollary \ref{Corollary collapses}.
\end{proof}

\section{Eigenvalue interlacing of relative Laplacians}
\label{section 5}
For a simplicial complex $(K,w_{K})$ and a subcomplex $(L,w_{L})$
which is pure of dimension $n$, such that $w_{L}(F)=w_{K}(F)$ for
every $F\in L$, we define the relative Laplacian $\mathcal{L}_{n}(K,L;\rb)$ as in  Section \ref{section 1}.

Let $\pi_{n+1}: C^{n+1}(K,\rb)\rightarrow C^{n+1}(K,L;\rb)$ be a projection map, i.e.,
$$
\pi(e_{[\bar{F}]})=\left\{\begin{array}{lr}
e_{[\bar{F}]} & \textrm{ if } \bar{F}\notin L,\\
0  &  \textrm{ otherwise, } 
\end{array}
\right.
$$ and let $\pi^{*}_{n}: C^{n}(K,L; \rb)\rightarrow C^{n}(K,\rb)$ be the  adjoint of the projection map defined on $n$-cochains, i.e.
$\pi^{*}(e_{[F]})=e_{[F]}$. Hence $\pi^{*}$ is an inclusion map.
Since $C^{n+1}(K,L; \rb)=C^{n+1}(K,\rb)$, then $\pi_{n+1}=id$, and  the following diagram commutes
$$
\begin{CD}
C^{n+1}(K,\rb)  @<\delta_{K}<<  C^{n}(K,\rb) \\
 @VV id V    @AA\pi^{*}A \\
C^{n+1}(K,L;\rb) @<\delta_{K,L}<<   C^{n}(K,L;\rb) \\
\end{CD}.
$$
From the commutativity of this diagram  and  (\ref{decomposition}), we
have
\begin{align}
\mathcal{R}_{\mathcal{L}_{n}(K,L)}(g)=&  \mathcal{R}_{id^{*}id}(\delta_{K}\pi^{*}g)   \mathcal{R}_{\mathcal{L}_{n}(K)}(\pi^{*}g)\mathcal{R}_{\pi\pi^{*}}(g)\\
=&  \mathcal{R}_{\mathcal{L}_{n}(K)}(\pi^{*}g)\mathcal{R}_{\pi\pi^{*}}(g).
\end{align}
Let $\lambda_{1}\leq \ldots \lambda_{N_{K}}$ be the eigenvalues of
$\mathcal{L}^{up}_{n}(K)$ and $\theta_{1}\leq \ldots
\theta_{N_{L}}$the  eigenvalues of $\mathcal{L}^{up}_{n}(K,L)$.
Since
$ \mathcal{R}_{\pi\pi^{*}}(g)=1 $, we have  
\begin{align}
\theta_{k}=\min_{\mathcal{V}_{k}}\max_{g\in \mathcal{V}_{k}}\mathcal{R}_{\mathcal{L}_{n}(K,L)}(g)=& \min_{\mathcal{V}_{k}}\max_{g\in \mathcal{V}_{k}}  \mathcal{R}_{\mathcal{L}_{n}(K)}(\pi^{*}g)\mathcal{R}_{\pi\pi^{*}}(g)\\
\geq & \min_{\mathcal{V}_{k}}\max_{g\in \mathcal{V}_{k}} \mathcal{R}_{\mathcal{L}_{n}(K)}(\pi^{*}g)\\
\geq & \lambda_{k},\\
\end{align}
and this is the lower interlacing inequality among the eigenvalues of $\mathcal{L}^{up}_{n}(K)$ and $\mathcal{L}^{up}_{n}(K,L)$.
The upper interlacing inequality  is
\begin{align}
\label{upper}
\theta_{k}=\min_{\mathcal{V}_{N_{L}-k+1}}\max_{g\in \mathcal{V}_{N_{L}-k+1}}\mathcal{R}_{\mathcal{L}_{n}(K,L)}(g)=& \min_{\mathcal{V}_{N_{L}-k+1}}\max_{g\in \mathcal{V}_{N_{L}-k+1}}   \mathcal{R}_{\mathcal{L}_{n}(K)}(\pi^{*}g)\\
=& \min_{\mathcal{V}_{N_{L}-k+1}}\max_{g\in \mathcal{V}_{N_{L}-k+1}} \mathcal{R}_{\mathcal{L}_{n}(K)}(\pi^{*}g)\\
\leq & \min_{\mathcal{V}_{N_{L}-k+1}}\max_{g\in \mathcal{V}_{N_{L}-k+1}} \mathcal{R}_{\mathcal{L}_{n}(K)}(g)\\
\leq & \min_{\mathcal{V}_{N_{K}-N_{K}+N_{L}-k+1}}\max_{g\in \mathcal{V}_{N_{K}-N_{K}+N_{L}-k+1}} \mathcal{R}_{\mathcal{L}_{n}(K)}(g)\\
\leq & \lambda_{k+N_{K}-N_{L}}.
\end{align}
We collect our results in the  following theorem.
\begin{thm}
Let  $(K,w_{K})$ be a simplicial complex and $(L,w_{L})$  a subcomplex, which is pure of dimension $n$, such that $w_{L}(F)=w_{K}(F)$ for every $F\in L$. Let 
$\lambda_{1}\leq \lambda_{2}\leq\ldots\leq\lambda_{N_{K}}$ and
$\theta_{1}\leq \theta_{2}\leq\ldots\leq\theta_{N_{L}}$ be the eigenvalues of 
$\mathcal{L}^{up}_{n}(K)$ and $\mathcal{L}^{up}_{n}(K,L)$, respectively. Then,
\begin{equation}
\lambda_{k}\leq \theta_{k}\leq  \lambda_{k+N_{K}-N_{L}},
\end{equation} 
where $N=\lambda_{N_{K}+1}=\ldots=\lambda_{2N_{K}-N_{L}}$, and $N$ is the number of vertices of $K$.
\end{thm}
\begin{rem}
It is not difficult to see that the matrix of the relative Laplacian $\mathcal{L}^{up}_{n}(K,L)$ is obtained by deleting rows and columns from the matrix
of $\mathcal{L}^{up}_{n}(K)$, thus one can apply  the  Cauchy interlacing  theorem  and obtain the same results.
However, the method employed here is general and can be used to treat
a variety of different interlacing problems, and the Cauchy interlacing theorem is just one of its special cases.
\end{rem}

\bibliography{interlacing}
\bibliographystyle{plain}{}

\end{document}